\documentclass{amsart}
\usepackage{amsfonts,amsmath,amssymb} 
\newtheorem{theorem}{Theorem}[section]
\newtheorem{lemma}[theorem]{Lemma}
\newtheorem{proposition}[theorem]{Proposition}

\newtheorem{corollary}[theorem]{Corollary}
\usepackage{bbm}
\usepackage{mathrsfs}

\theoremstyle{definition}
\newtheorem{definition}[theorem]{Definition}

\theoremstyle{remark}

\numberwithin{equation}{section}

\newcommand{\I}{\mathbbm{1}}

\newcommand{\Ndb}{{\mathbb N}}
\newcommand{\Cdb}{{\mathbb C}}
\newcommand{\Tdb}{{\mathbb T}}
\newcommand{\Ddb}{{\mathbb D}}

\begin{document}
\title[Characters for noncommutative function algebras]{On vector-valued characters for noncommutative function algebras} 
\author{David P. Blecher}
\address{Department of Mathematics, University of Houston, Houston, TX
77204-3008}
\email[David P. Blecher]{dpbleche@central.uh.edu}
\author{Louis E. Labuschagne}
\address{DSI-NRF CoE in Math. and Stat. Sci, Pure and
Applied Analytics,\\ Internal Box 209, School of Math. \& Stat. Sci.,
NWU, PVT. BAG X6001, 2520 Potchefstroom, South Africa}
\email{louis.labuschagne@nwu.ac.za}

\date{\today}
\thanks{DB  is supported by a Simons Foundation Collaboration Grant. LL is supported by the National Research Foundation (IPRR Grant 96128 and KIC Grant 171014265824). Any opinion, findings and conclusions or recommendations expressed in this material, are those of the author, and therefore the NRF do not accept any liability in regard thereto.} 

\begin{abstract}  Let $A$ be a closed subalgebra of a $C^*$-algebra, that is a norm-closed algebra of Hilbert space 
operators.    We generalize to such operator algebras  several  
key theorems and concepts from the theory of classical function algebras.    In particular we consider several 
problems  that arise when generalizing classical function algebra results involving characters (nontrivial
 homomorphisms from the algebra into 
the scalars). 
For example, the Jensen inequality, the related Bishop-Ito-Schreiber theorem, and 
the theory of Gleason parts.  
Inspired by Arveson's work on noncommutative Hardy spaces, we 
replace characters (classical function algebra case) by  
 $D$-{\em characters}; certain completely contractive 
homomorphisms $\Phi : A \to D$, where $D$ is a $C^*$-subalgebra of $A$.  
Using Brown's measure  and a potential theoretic balayage argument we prove a partial  noncommutative 
Jensen inequality appropriate  for $C^*$-algebras with a tracial state.  We also show that 
this  Jensen inequality characterizes $D$-characters among the module maps.  Other advances include
a theory of noncommutative Gleason parts appropriate for $D$-characters, which uses Harris' noncommutative hyperbolic metric and 
Schwarz-Pick inequality, and other ingredients.  
As an application of Gleason parts we 
show that in the antisymmetric case, one is guaranteed the existence of a `quantum' Wermer embedding function, and also of non-trivial 
compact Hankel operators, whenever the Gleason part of the canonical trace is rich in tracial states.
\end{abstract}

\maketitle

\section{Introduction} 
In this paper an {\em operator algebra}
 is a unital closed subalgebra $A$ of a unital $C^*$-algebra $B$, or equivalently
 a norm-closed algebra of Hilbert space
operators containing the identity 
operator.   If $B$ is commutative then $A$ may be viewed as a  \emph{function  algebra}; that is a 
norm-closed unital subalgebra $A$ of $C(K)$, for a compact Hausdorff space $K$.   Indeed $A$ may be viewed as a  \emph{uniform algebra}, 
i.e.\ it also separates the points of $K$.   This may be done for example by taking $K$ to be the {\em maximal ideal space} of $A$. A  successful quest to develop a fuller
noncommutative theory of the classical theory of uniform algebras needs to 
identify appropriate operator algebraic variants of several fundamental concepts from 
the latter theory.   Some of the first of these concepts that comes to mind are
 the `characters' of $A$ (which constitute the maximal ideal space of $A$),  the `representing measures' for those characters, 
and also the `geometric mean' of the modulus of an element of $A$ computed with respect to some measure/state.
For the sake of clarity, we pause to review the appropriate classical concepts.   More details may be found in e.g.\ \cite{Gam}.

For a function algebra $A$ on $K$ as above, the multiplicative linear functionals on $A$ assuming the value 1 at the identity 
$\I$ are referred to as the \emph{characters}
of $A$. A positive measure $\mu$ on $K$ is called a \emph{representing measure} for a character $\Phi$ if
$\Phi(f)=\int_K \, f \,d\mu$ for all $f\in A$. The functional $\widetilde{\Phi}(g) = \int_K \, g \,d\mu$ on $C(K)$ is a state on $C(K)$,
and indeed representing measures for $\Phi$ are in a bijective correspondence with the state extensions to $C(K)$ of $\Phi$.
(We recall that a state is a contractive unital linear 
functional.)    We may thus regard the 
set of representing measures for a given character, as the set of all norm-preserving extensions to $C(K)$ of that character.

An obvious choice for  a noncommutative version of characters on an operator algebra $A$ would be the completely contractive unital homomorphisms
$\Phi : A \to B(H)$, for Hilbert spaces $H$.   In this case the `noncommutative representing measures' would be the $B(H)$-valued 
extensions of $\Phi$ to a containing $C^*$-algebra $B$, which are 
completely positive (or equivalently, in this case, completely contractive).    Such noncommutative representing measures $\Psi : B \to B(H)$ always exist,
by Arveson's extension theorem \cite[Theorem 1.2.9]{Arv2}.  However although these noncommutative notions are appropriate in  many settings
(notably, the theory of the noncommutative Shilov and Choquet boundary, see
e.g.\ \cite{Arv2,DK}), they do not facilitate the   generalization of some other important parts of the theory
of uniform algebras.   This is no doubt some 
part of the reason why some of the theory of uniform algebras has not been extended to operator algebras to date.
We launch an effort here to try to remedy some small part of
this, focusing on three or four classical topics involving characters. 
Another good  illustration of the sometimes inappropriateness of the above class of 
noncommutative characters appears in our sequel papers \cite{BLrm,BFZ} on noncommutative versions of the classical Hoffman-Rossi theorem.  The latter may be described as `the existence of weak* continuous representing 
measures'.  It asserts that if $A$ is a weak* closed subalgebra of a commutative $W^*$-algebra $M$, and $\Phi$ is a weak* continuous character on $A$, then  $\Phi$ has a  weak* continuous 
 state extension to $M$ (see e.g.\ \cite{HR} and   \cite[Theorem 3.2]{Gam}).   In \cite{BLrm,BFZ} it is noted that the $B(H)$-valued version of this result fails, and it is proved that there are satisfying 
(but sometimes deep) noncommutative versions which hold for the following smaller class of operator valued homomorphisms on $A$.  

In this paper our `noncommutative characters' on an operator algebra $A$
will usually be the contractive unital homomomorphisms
$\Phi : A \to D$ which are also $D$-module maps (that is, $\Phi(d_1ad_2)=d_1\Psi(a)d_2$ for all $a\in A$, $d_1, d_2\in D$), for a unital $C^*$-subalgebra $D$ of $A$.   We call these the $D$-{\em characters}.
They are necessarily completely contractive
(see the computation a couple of paragraphs above Lemma \ref{exucexp}).   The usual characters may be viewed as the case when $D$ is $\Cdb \I$, the scalar multiples
of the identity, 
and in this case the `noncommutative representing measures' would correspond to the state extensions  to $B$.
In a couple of sections of our paper we will restrict attention  further to the case that 
$A$ is a unital subalgebra of a $C^*$-algebra $B$, where $B$ possesses a fixed 
state or trace $\psi$, and that 
$\Phi$ is $\psi$-{\em preserving}: that is, $\psi \circ \Phi = \psi_{|A}$.    
In this setting one is particularly interested in  `noncommutative representing measures' $\Psi$ for $\Phi$ which are also $\psi$-preserving on $B$;
here one may consider `noncommutative representing measures' to be such a pair $(\Psi, \psi)$.
In this paper $\Psi$ can usually be taken to be $D$-valued, and often $\Psi$ is completely 
determined by $\psi$ and $D$.    Indeed in the classical case $\Phi = \psi(\cdot) \I$.
  A simple motivating noncommutative example: the canonical map 
from  the upper triangular matrices onto $D_n$, the matrices supported on the main diagonal,  is a $D_n$-character.
The associated representing measure pair here is the normalized matrix trace, and the (unique trace preserving)  conditional expectation of $M_n$ onto $D_n$.   

The underlying theme running throughout the paper are the 
natural questions surrounding $D$-characters and their representing measures, arising by 
comparison with the classical theory of uniform algebras (from e.g.\ \cite{Gam}).   We then systematically begin to 
answer these questions.  

There are several  motivations for considering our $D$-characters, besides the fact that
they allow generalization of certain aspects of the classical theory that were previously 
intractable.  
Historically,
$D$-characters played a crucial role in one of the 
most important early papers in the subject of 
nonselfadjoint operator algebras/noncommutative function algebras, Arveson's 
seminal `Analyticity in operator algebras' \cite{Arv}.  
This work initiated a program of using a particular class of $D$-characters,  
von Neumann algebras, and noncommutative $L^p$ spaces, to give 
a vast generalization of the theory of Hardy's $H^p$ spaces.  In the classical theory of uniform 
algebras such `generalized Hardy space' theory takes up a large part of the standard texts.
Arveson gave very many interesting examples of $D$-characters, which also serve as motivating examples for us.
Subsequent to the verification of the noncommutative Szego formula \cite{LL1}, his {\em subdiagonal} setting 
has been developed further by very many authors (see e.g.\ \cite{BLsurv,Ueda} and references therein), 
and is still a very active research area internationally.     Our setting here 
is  much more general than Arveson's, but may be viewed as using his idea to more systematically
understand/transfer some parts of  function algebra theory to noncommutative settings.  

Another motivation for our generalization of characters  
is so that one may tie in to the powerful theory of von Neumann algebraic
conditional expectations.  Indeed the latter may 
be viewed within the setting above of our $D$-characters in the case that $D$ is a von Neumann subalgebra of a von Neumann algebra $B$, 
and $A = D$, and $\Phi$ is the identity map on $D$.  Here `noncommutative representing measures' $\Psi : B \to D$ 
 are necessarily conditional expectations by Tomiyama's result (\cite{tom},
 see p. 132--133 in \cite{Bla} for the main facts about conditional expectations and their relation to bimodule maps and projections maps of norm 1).
The existence of such `representing measures' becomes the important issue of existence of conditional expectations onto a von Neumann subalgebra.
This perspective is highlighted in the sequel  \cite{BLrm}, but also plays 
a major role in several sections of the present paper.   

Given a state $\psi$ on $B$, a geometric mean computed with respect to $\psi$ may be gleaned 
from the theory of the Fuglede-Kadison determinant. Specifically, the $\psi$-geometric mean 
$\Delta_\psi(b)$ of $b\in B_+$ is defined to be the quantity $\exp (\psi(\ln b))$ if $b$ is invertible, with $\lim_{\epsilon \to 0^+} \, \exp \psi (\ln (b + \epsilon \, \I))$ being used if $b$ is not invertible. 
For a tracial state $\tau$ on $B$  the quantity $\Delta_\tau$ 
has almost all the usual properties of the Kadison-Fuglede determinant \cite{FK}, as we show in Lemma \ref{FK}.
It also generalizes the classical quantity $\exp ( \int_K \, \ln |f| \, d \mu)$ found frequently  in the theory of uniform algebras, sometimes in logarithmic form  $\int_K \, \ln |f| \, d \mu$. 
 Indeed if we have a state $\tau$ on $C(K)$, let $\mu$ be the probability measure on $K$ corresponding to $\tau$.  Then $$\Delta_\tau(f) = \lim_{n \to \infty} \, \exp \left( \int_K \, \ln \left(|f| + \frac{1}{n}\right) \, d \mu\right)
 = \exp \left( \int_K \, \ln |f| \, d \mu\right), \qquad f \in C(K),$$
using Lebesgue's monotone convergence theorem. (We note that 
Lebesgue's monotone convergence theorem for decreasing functions works even if functions are allowed to take negative values, provided that their integrals are not $+\infty$.) 

Classically, a representing measure $\widetilde{\Phi}$ for a character $\Phi$ of a uniform algebra $A$ is said to be 
a {\em Jensen measure} if it satisfies {\em Jensen's inequality} on $A$, 
namely $|\Phi(a)|\leq \Delta_{\widetilde{\Phi}}(a)$ for all $a\in A$. Jensen measures are an important feature  
in the classical theory of uniform algebras, 
 being the dual  notion to {\em subharmonicity}. 
The main result 
in Section \ref{jensen}   is a partial  noncommutative 
Jensen inequality appropriate  for $C^*$-algebras with a tracial state. The key tools in this development are Brown's measure \cite{Brown,HS}, and a potential theoretic balayage argument.  
We also need here and in some other sections a reduction to the finite von Neumann algebra case using a method 
which we develop later in Section 1.   
Section \ref{GW} presents some $D$-valued variants of the classical Gleason-Whitney theorem, which concerns existence  and uniqueness of weak* continuous 
representing measures (in their case the algebra is $H^\infty(\Ddb)$).   

In the classical case the existence of a Jensen measure actually  characterizes the characters among the linear functionals 
on a uniform algebra.  This is due to Bishop and Ito-Schreiber
(see e.g.\ \cite{Sid}).  In  Section \ref{BIS} we present noncommutative variants of this result, for example that
a special case of the 
Jensen inequality characterizes $D$-characters among
the $D$-module maps. 
One complication that occurs in the noncommutative case here
is taken care of by a new result concerning the question of when in a certain setting a Jordan homomorphism is an algebra homomorphism.

An important classical feature of uniform algebras is that the `Gleason relation' 
$\| \varphi - \psi \| < 2$ turns out to be equivalent to boundedness of 
the hyperbolic distance between values of $\varphi$ and $\psi$ on the ball.  This
is an equivalence relation on the set of characters (resp.\ the set of representing measures of characters).
Hence they may be partitioned into equivalence classes, called \emph{(Gleason) parts}.    
The Gleason relation  does not seem to have a $B(H)$-valued analogue suitable for our 
  purposes, but we will show in Section \ref{gp} that interestingly  it does have 
a noncommutative variant for our $D$-characters.   Thus we initiate a theory of noncommutative Gleason parts, appropriate for $D$-characters, which uses Harris' noncommutative hyperbolic metric and 
Schwarz-Pick inequality \cite{HarHol} and other ingredients.   
Two $D$-characters are  Gleason equivalent if they are Harnack equivalent in the sense of \cite{SV}, and we prove the
converse in the scalar case.  

Finally in Section \ref{hankel} we present a surprising and important application of 
the above new theory of Gleason parts. We pause to justify the 
significance of the results in this section. The 
study of  Toeplitz operators on $H^2$-spaces of maximal subdiagonal subalgebras is receiving renewed interest at present
(see for example \cite{Bek, Pru}. In the classical setting of $H^2(\mathbb{D})$, the index theory of Fredholm Toeplitz operators is a particularly elegant and deep part of this theory. (See for example the classical result of Gohberg and Krein \cite{GK}, which connects the index of  Fredholm Toeplitz operator $T_f$, with the winding number of the symbol $f$. More modern developments on this topic may be found in \cite{Mur, Mir} for example.) If both Hankel $\mathcal{H}_f$ and Toeplitz maps $T_f$ corresponding to some symbol $f\in L^\infty$ are regarded as maps from $H^2$ into $L^2$, then $T_f= I -\mathcal{H}_f$. The map $T_f$ will therefore clearly be semi-Fredholm whenever $\mathcal{H}_f$ is compact. When attempting to generalize this theory to more general contexts, possibilities arise that do not appear in the classical setting. Note for example that 
there are group algebraic contexts which admit no non-trivial compact Hankel matrices whatsoever! (See the discussion preceding Definition 10 in \cite{Exel}.) Hence the identification of criteria which guarantee the existence of compact Hankel maps (and, by 
extension, of semi-Fredholm Toeplitz maps) forms an important part of the generalized theory. It is precisely on this matter that the theory of Gleason parts proves to be crucial. Specifically we are able to provide sufficient conditions formulated in terms of Gleason 
parts for the existence of a so-called Wermer embedding function in the noncommutative context. Building on ideas of \cite{CMNX}, one may show that the existence of such a function in turn guarantees the existence of non-trivial compact Hankel operators. In the 
commutative context this relation is very precise (see \cite[Theorem 2]{CMNX}.)

We now describe the setting for most of our paper.
Throughout $A$ is a unital operator algebra,  a unital subalgebra of a $C^*$-algebra or von Neumann algebra $B$. 
The difference between the function  algebra and uniform  algebra notions defined at the start of our paper is that the latter generates
the containing $C^*$-algebra $C(K)$.  This follows from the Stone-Weierstrass theorem.
In this paper however we shall not  require that
$A$ generates $B$ as a $C^*$-algebra.  Thus the difference between function  algebra and uniform  algebras
will not play a role in our paper we use the two terms somewhat interchangeably.

As we said, we generally replace characters (classical function algebra case) by the $D$-characters
above.   These are the contractive
homomorphisms $\Phi : A \to D$, where $D$ is a $C^*$-subalgebra of $A$
containing $\I_A$, which restrict to the identity map on $D$.
Any such map is a $D$-module map and is completely contractive.     That is,
$\| \Phi \|_{\rm cb} = \| \Phi \|$ = 1.
This may be proved by a standard trick.
Namely, we recall that if $[x_{ij}] \in M_n(D)$ then $\| [x_{ij} ] \|_n = \sup \{ \| \sum_{i,j} s_i^* x_{ij} r_j \| \}$,
the supremum taken over $s_i, r_i \in D$ with $\sum_i \, r_i^* r_i$ and $\sum_i \, s_i^* s_i$ both contractive.
Indeed this is clear for example by viewing $M_n(D)$ as the adjointable maps on the (right) $C^*$-module sum
of $n$ copies of $D$.   Hence  if $[x_{ij}] \in {\rm Ball}(M_n(A))$, then
$\| [\Phi(x_{ij}) ] \|$ is dominated by the supremum over such $r_i, s_i$ of
$$\| \sum_{i,j} s_i^* \Phi(x_{ij}) r_j \| = \| \Phi(x)\| \leq \|  \Phi \|,$$
where $x = \sum_{i,j} s_i^* x_{ij} r_j$, which is a contraction in $A$.  Hence
$\| \Phi \|_{\rm cb} = \| \Phi  \|$.
Conversely, we recall
that any completely contractive unital idempotent map $\Phi$ on $A$ with range
$D$ is a $D$-module map \cite[Proposition 5.1]{BLI}.

The following facts and notation will be used in many places in our paper to
reduce to the finite von Neumann algebra case.
If $\tau$ is a normal (that is, weak* continuous)
tracial state  on a von Neumann algebra $M$ which is not faithful, then the left kernel of $\tau$
equals the right kernel.  This is therefore
a weak* closed ideal $J$ of $M$ contained in Ker$(\tau)$.
There is therefore a central projection
$z \in M$ with $J = z^\perp M$.
The quotient of $M$ by $J$ may also be identified with $N = z M$.
Then $\tau' = \tau_{|N}$ is a faithful normal tracial state  on the von
Neumann algebra $N$.   We will often simply write $\tau$ for $\tau'$.   Since $z$ is the support of $\tau$, we have that $\tau(z) = 1$
and $\tau(x) = \tau(zx)$ for all $x \in M.$
The following is no doubt known to a few experts, but we were unable to find it in the literature.

\begin{lemma}  \label{exucexp} If $D$ is a von Neumann subalgebra of a von Neumann algebra $M$, and if $\tau$ is a normal tracial state on $M$ which is faithful on $D$ then there is a unique normal conditional expectation $\Psi : M \to D$ that is $\tau$-preserving (i.e., $\tau \circ \Psi =
 \tau$).        \end{lemma}

\begin{proof} Let $z$ be the projection above.
If $\tau$ is faithful on $D$ then multiplication by $z$ on $D$ is a faithful $*$-homomorphism onto $Dz$,
since $dz = 0$ implies $\tau(d^* dz) = \tau(d^*d) = 0$, and $d = 0$.

Note that $Dz$ is a von Neumann subalgebra of $Mz$, so that there is
a unique $\tau$-preserving normal conditional expectation $\rho : Mz \to Dz$.   Define $\Psi : M \to D$ by $\Psi(m) z = \rho(mz)$ for $m \in M$.
This is well defined by the fact noted in the last paragraph.   Hence $\Psi$ is linear, indeed it is
a $D$-bimodule map (since e.g.\ $\rho(dmz) = dz \rho(mz) = d \rho(mz)$). 
Since $$\tau(\Psi(m)) = \tau(\Psi(m)z) =\tau(\rho(mz)) =\tau(mz) = \tau(b),$$
 $\Psi$ is $\tau$-preserving.  Also $\Psi(\I) = 1$.  
Now $$\tau(\rho(az) ) = \tau(\Phi(a)) = \tau(a) = \tau(az),$$ and so $\rho$ is $\tau$-preserving. 

 To see that $\Psi$ is positive let $m \in M_+$. Then
 $\rho(mz) \geq 0$.   If $d \in D_+$ then 
 $$\tau(d \Psi(m)) = \tau(d \Psi(m)z) = \tau(d \rho(mz)) \geq 0 .$$
 So $\Psi$ is positive.  
 If $\rho$ is normal 
 then so is $\Psi$ since given a bounded net $m_t \to m$ weak* in $M$ we have that
 $$\Psi(m_t) z = \rho(m_tz) \to \rho(mz) = \Psi(m) z .$$
Using the easily verifiable fact that in the present setting $L^1(D,\tau)\subset L^1(M,\tau)$, it is clear that $$\tau(d \Psi(m_t)) = \tau(d \Psi(m_t)z)
 \to \tau(d \Psi(m)z) = \tau(d \Psi(m)), \qquad 
 d \in L^1(D,\tau) .
 $$
 Thus $\Psi(m_t) \to \Psi(m)$ weak* in $D$.

The uniqueness follows from the proof of \cite[Lemma 5.3]{BLI}.
\end{proof}

{\bf Remark.}   We give an alternative argument for part of the last proof, that has some ideas that may be useful elsewhere.  Let $\Psi' : M \to D$ be another $\tau$-preserving normal conditional expectation $M \to D$ extending $\Phi$.
Then $\rho' : mz \mapsto \Psi'(m) z$ maps $Mz \to Dz$.   If $m z = m'z$ then 
$$\tau(d \Psi'(m-m')) = \tau(d(m-m')) = \tau(d(m-m')z) = 0 , \qquad d \in D.$$
Thus $\Psi'(m-m') z = 0$, so that $\rho'$ is well defined.   Now $$\tau(\rho(mz) ) = \tau(\Psi'(m)) = \tau(m)
= \tau(mz),$$
so $\rho'$ is $\tau$-preserving.  It is clearly a $Dz$-bimodule map onto $Dz$, and idempotent.  Finally, $\rho'$ is positive by a slight variant of the argument above that $\Psi$ is positive:
if $mz \geq 0$ then
 $$\tau(d \Psi'(m)z) = \tau(\Psi'(dm)) = \tau(dm)
 = \tau(dz  \rho(mz)) \geq 0 , \qquad d \in D_+.$$
Thus $\rho' = \rho$ by the uniqueness of  $\tau$-preserving normal conditional expectations
on a finite von Neumann algebra.
So $\Psi'(m) z = \Psi'(m) z$ and
$\Psi'(m) = \Psi'(m)$.   Thus $\Psi' = \Psi$.

\medskip

We now move towards defining a determinant for a tracial state  which is not faithful.

\begin{lemma} \label{zcalc} If $x$ is a positive invertible element in a unital
$C^*$-algebra $B$, and if $z$ is a central projection in $B$, then for any continuous 
 function $f$ on $\sigma(x)$ we have that $f(xz)=f(x)z$, where $f(xz)$ is  computed in 
$Bz$. In particular, 
 if $p > 0$ then 
 $\ln (zx) = z \ln x$ and $(xz)^p = x^p z$.   Here $\ln(zx)$ and $(xz)^p$ are computed in 
$Bz$. 
\end{lemma}

\begin{proof}   Let $\pi : B \to Bz$ be multiplication by $z$, a surjective unital $*$-homomorphism. 
We have $\sigma_{Bz}(zx) \subset \sigma_B(x)$, so that $f(xz)$ makes sense in $Bz$. 
By a well known property of the functional calculus $f(zx) = f(\pi(x)) =  \pi(f(x)) = f(x)z$. \end{proof}

For $M, \tau, N, z, \tau' = \tau_{|N}$ as above Lemma \ref{exucexp}, let $\Delta = \Delta_{\tau'}$ be the 
Fuglede-Kadison determinant for $N$.  A Brown measure and determinant $\Delta_\tau$ 
on $M$ with respect to $\tau$ may be defined via $\tau'$ and the associated  determinant on $N$: $\Delta_\tau(a) = \Delta(z a)$ for $a \in M$.  
This will have most of the usual properties of the determinant (see \cite{FK,Brown,Arv, HS,BLsurv}), as we prove below.  In fact many of the properties below 
in the von Neumann algebra case are already noted in \cite{Arv} (although the determinant is defined there by  
the formula involving $\exp, \tau,$ and $\ln$).

 If  $\tau$ is a tracial state on a $C^*$-algebra $B$,  then $\tau^{**}$ is a
normal tracial state on $M = B^{**}$.   As we did above Lemma \ref{exucexp}, we may take the quotient by the 
left kernel $J = B^{**} z^{\perp}$, to get a faithful normal tracial state   $\tau'$ on $N = B^{**} z$.   Then 
$N = B^{**} z$ has a Kadison-Fuglede determinant $\Delta$, as above, and we may define a  determinant on $B$ by $\Delta_\tau(x) = \Delta(zx)$, for $x \in B$.    This again seems to have almost all the usual properties of the determinant: 

\begin{lemma} \label{FK} If $\tau$ is a normal tracial state on a von Neumann algebra $M$
(resp.\ is  a  tracial state on a $C^*$-algebra $B$), and if
$\Delta_\tau$ is the determinant defined above, then:
\begin{itemize} \item $\Delta_\tau(a) \leq \tau(|a|).$
	\item $\Delta_\tau(a b) = \Delta_\tau(a) \Delta_\tau(b)$.
	\item $\Delta_\tau(\lambda u) = |\lambda|$ for any unitary in $M$ (resp.\ $B$) 
	and $\lambda \in \Cdb$.
	\item $\Delta_\tau (\exp(a)) = |\exp ( \tau(a))|$.
	\item $\Delta_\tau(a^*) =  \Delta_\tau(a) = \Delta_\tau(|a|)$.
	\item 
$\Delta_\tau(a) = \exp (\tau(\ln |a|))$ if $a$ is invertible in $M$ (resp.\ $B$), otherwise 
$\Delta_\tau(a) = \lim_{\epsilon \to 0^+} \, \exp \tau (\ln (|a| + \epsilon \I))$ (in the $B$ case we assume that $B$ is unital here).
\item If $0 \leq a \leq b$  
 then $\Delta_\tau(a) \leq \Delta_\tau( b)$, and $\Delta_\tau(a^p) = \Delta_\tau(a)^p$ if $p > 0$.
\item $\Delta_\tau$ is upper-semicontinuous in the norm topology.
\item If $a$ is invertible in $M$  (resp.\ $B$) 
then $\Delta_\tau(a) = \lim_{\epsilon \to 0^+} \, \tau(|a|^\epsilon )^{\frac{\I}{\epsilon}}.$ 
\end{itemize}
Here $a, b \in M$  (resp.\ $B$).
\end{lemma}   
\begin{proof} 
The ideas for the proofs of nearly all of these occur in the proof of the first, namely
$$\Delta_\tau(a) = \Delta(az) \leq \tau(|az|)
= \tau(|a|z) = 
\tau(|a|),$$ 
where for the inequality, we used the analogous property for the usual Fuglede-Kadison determinant.
In the case for $B$ one should replace $\tau$ by $\tau^{**}$, and note that $\tau^{**}(|a|) = \tau(|a|)$.  
 Similarly $\Delta_\tau$ has the multiplicativity property of the determinant since $\Delta$ does, in the case for $B$ this reads
$$\Delta_\tau(ab) = \Delta(abz) =  \Delta(az)  \Delta(bz) = \Delta_\tau(a) \Delta_\tau(b) .$$
  Clearly $\Delta_\tau(\I) = \Delta(z) = \Delta(\I_N) = 1$.   Thus if $u$ is a unitary 
we have $\Delta_\tau(u) = \Delta(zu) = 1$ since $zu$ is a unitary in $N$.     That $\Delta_\tau(\lambda u) = |\lambda|$ is now easy.
If $a$ is invertible in $M$ then 
$$\Delta_\tau(a) = \Delta(za) = 
\exp (\tau(\ln |za|)) = \exp (\tau(z \, \ln |a|)) = \exp (\tau(\ln |a|)).$$ 
For general $a \in M$ we have $$\exp \tau (\ln (|a| + \epsilon)) = \exp \tau'( z (\ln (|a| + \epsilon))) = \exp \tau' ( \ln (| za | + \epsilon \I_N)) ,  \qquad a \in M.$$   
Letting $\epsilon \to 0$ we see that $$\Delta_\tau(a)  = \Delta(za) = \lim_{\epsilon \to 0^+} \, \exp \tau (\ln (|za| + \epsilon z)) = 
\lim_{\epsilon \to 0^+} \, \exp \tau (\ln (|a| + \epsilon)).$$  
In the case of $x \in B$, we may use Lemma \ref{zcalc} to see that if $x\in B$ is invertible, then we then have 
$$\Delta_\tau(x) = \Delta(x z) = \exp(\tau^{**}(\ln |\hat{x} z|))= \exp(\tau^{**}(z \,  \widehat{\ln \, |x|})) = \exp(\tau(\ln |x|)).$$ 
Here we used the fact that  $\ln (\hat{|x|}) = \widehat{\ln \, |x|}$, which follows by the last line of the proof of  Lemma \ref{zcalc}
with $\pi$ replaced by the canonical  $*$-homomorphism $\hat{\cdot} \,  : B \to B^{**}$.  
For arbitrary $x \in B$ we have by the above that
$$\Delta_\tau(x) = \Delta(x z) = \lim_{\epsilon \to 0^+} \,  \Delta_\tau(|x| + \epsilon)) = \lim_{\epsilon \to 0^+} \, \exp \tau (\ln (|x| + \epsilon \I)). $$
The remaining assertions follow by similar considerations.  For example, 
the statements involving $a^p, |a|^\epsilon,$ and $\exp (a)$ follow similarly from the analogous properties for the usual Fuglede-Kadison determinant 
(see e.g.\  \cite{FK,Brown,Arv, HS,BLsurv}) and the
relations $(za)^p = z a^p$ and 
$z \exp (a) = \exp(az)$ from Lemma \ref{zcalc}, the latter exponential computed in $Mz$. 
\end{proof}

\section{The Jensen inequality} \label{jensen}

In \cite[Proposition 4.4.4]{Arv} Arveson showed that for his algebras the Jensen
inequality $$\Delta(\Phi(f)) \leq \Delta(f), \qquad f \in A,$$
was equivalent to the statement $\Delta(\I + x) \geq 1$ if $x \in A$ with $\Phi(x) =
0$. Clearly the Jensen inequality implies the last statement, since $\Phi(\I+x) = 1$.  
It also implies  the {\em Jensen
equality} $\Delta(\Phi(f)) = \Delta(f)$ for  elements $f$ which are invertible in $A.$ 
Taking our cue from these facts, we introduce the following notions.
We say that $\Phi$ satisfies the {\em ball-Jensen inequality}  if 
$$\Delta(\I + x) \geq 1, \qquad
\text{for all} \; x \in A \; \; \text{such that} \; \Phi(x) = 0 \, \; \text{and} \;  r(x) \leq 1.$$ 
Here $r(x)$ is the spectral radius.    
Clearly the Jensen inequality implies the ball-Jensen inequality.   
We will prove that in our algebras below we always have 
the {\em ball-Jensen equality}, namely: 
$$\Delta(\I + x) = 1, \qquad
\text{for all} \; x \in A \; \; \text{such that} \; \Phi(x) = 0 \, \; \text{and} \;  r(x) \leq 1.$$  

  Recall that a map $\Phi$ on $A$ is $\tau$-preserving if $\tau \circ \Phi =
 \tau_{|A}$.

\begin{theorem} \label{baly}  Let $M$ be a $C^*$-algebra with a tracial state $\tau$. 
  \begin{itemize}
\item [{\rm (1)}]  If $x \in M$ with $r(x) \leq 1$ and $\tau(x^m) = 0$ for all $m \in \Ndb$  then $\Delta_\tau(\I + x) = 1$. 
Equivalently, if $a \in M$ with  $r( \I - a ) \leq 1$ 
and $\tau(a^m) = 1$ for all $m \in \Ndb$ 
then $\Delta_\tau(a) = 1$.   
\item [{\rm (2)}]   Let $\Phi : A \to D$ be a $\tau$-preserving  unital homomorphism from a subalgebra $A$ of $M$
 onto a  $C^*$-subalgebra
$D$ of $A$, such that $\Phi$ is the identity map on $D$.  Then  $A$ satisfies the ball-Jensen equality
$\Delta_\tau (\I + x) \geq 1$ for $x \in A$ such that $\Phi(x) = 0$ and $r(x) \leq 1.$ Also, if
$x \in A$ with 
 $\Phi(x)$ invertible in $D$, and if $\|  \Phi(x)^{-1} x - \I \| \leq 1$, then  $\Delta_\tau (\Phi(x))   = \Delta_\tau (x)$.
\end{itemize} 
If $\Phi : A \to D$ is as  in {\rm (2)} and $\tau$ is faithful  
 then we have $D = A \cap A^*$.   Such $\Phi$ as in {\rm (2)} is unique if $\tau$ is faithful.
 Indeed if  $M$ is a von Neumann algebra, $\tau$ is faithful,  and $D$ is weak* closed, then $\Phi$
is the restriction to $A$ of the canonical $\tau$-preserving  normal  faithful conditional expectation of
$M$ onto $D$. \end{theorem}

 \begin{proof}  If  $\Phi$ is a $\tau$-preserving 
homomorphism and $\tau$ is faithful then  it follows by an argument of Arveson \cite{Arv} that $D = A \cap
A^*$:  Certainly $D \subset A \cap A^*$.
If $x \in (A \cap A^*)_{\rm sa}$ then
$\Phi((x - \Phi(x))^2) = (\Phi(x - \Phi(x)))^2 = 0$.   Applying $\tau$ we have
$\tau((x - \Phi(x))^2) =0$ so that $(x - \Phi(x))^2
= 0$.  Hence $x - \Phi(x) =0$, so $x = \Phi(x) \in D$.    Thus $D = A \cap A^*$.
The uniqueness of $\Phi$ if  $\tau$ is faithful follows from  \cite[Lemma 5.3]{BLI}.  Because of this uniqueness, 
 if $M$ is a von Neumann algebra and $D$ is weak* closed then $\Phi$
is the restriction to $A$ of the canonical $\tau$-preserving faithful conditional expectation of
$M$ onto $D$ (namely the dual of the embedding 
of $L^1(D,\tau)$ in $L^1(M,\tau)$).  

 Next we prove (1) and (2)  in the case that $M$ is a von Neumann algebra with a  faithful normal tracial state $\tau$.   
 We consistently write $\Delta_\tau$ as $\Delta$ in this case.  For (1), 
let $\mu$ be the Brown measure of $x$.  By for example\ the last assertions of \cite[Theorem 3.13]{Brown} we first see that 
  $\int_{\Cdb} \, z^n \, d \mu(z) = 0$ for all $n \in \Ndb$,
and, second, 
that  $\Delta(1 + x) \geq 1$ if 
$\int_{\Cdb} \, \, \ln \, |1+z| \, d \mu(z) \geq 0$. 
Write $\mu =  \mu_0 + \mu_1$
 where $\mu_1 = \mu_{| \Tdb}$ (the part of $\mu$ supported on $\Tdb$),  and 
$\mu_0 = \mu_{| \Ddb}$ where $\Ddb$ is the open unit disk. The moments of $\mu$ are the sum of the moments of $\mu_0$ and the moments 
of $\mu_1$.
The balayage $\nu_0$ of $\mu_0$ onto $\Tdb$  (see  \cite[Theorem
II.4.1 or II.4.7]{SaffTotik}) is a positive measure on the circle whose moments
are by (c) in the cited theorem,the same as the moments of $\mu_0$.   (We are grateful to Brian Simanek 
for help with this balayage argument.)    Let $\nu =\nu_0 + \mu_1$.
So the moments of $\mu$ are the same as the moments of $\nu$.
So $\int_{\Cdb} \, z^n \, d \nu = \int_{\Tdb} \, z^n \, d \nu = 0$ for all positive integers $n$, and hence for all  integers by taking complex conjugates.
Hence $\nu$ is a multiple of arc length measure on $\Tdb$.  Indeed the $*$-moments determine
a measure.  

The potential of $\mu$, that is $-\int_{\Cdb} \,
\ln |z-y| \, d \mu(y)$, is the sum of the potentials of $\mu_0$ and $\mu_1$.   These potentials lie in $(-\infty, \infty]$.   
 By a fundamental fact about balayage from the last cited theorem from \cite{SaffTotik}, the potential of $\mu_0$ on $\Tdb$ equals
the potential of $\nu_0$ on $\Tdb$.  This theorem mentions {\em regular boundary points}, defined on p.\ 54 of that 
reference, which in the case we need, are all the points on $\Tdb$ (by e.g.\ the line after Theorem 4.6 on p.\ 54 of \cite{SaffTotik}). 
Thus the potential of $\mu$ on $\Tdb$ equals
the potential of $\nu$. That is, $-\int_{\Cdb} \, \ln |z-y| \, d \nu(y) = -\int_{\Cdb} \,
\ln |z-y| \, d \mu(y)$ for
$z$ on and outside $\Tdb$.  Setting $z = -1$ we have 
$$\int_{\Cdb} \, \ln |1+y|
\, d \mu(y) = \int_{\Cdb} \, \ln |1+y| \, d \nu(y)   \geq 0.$$
Indeed $$\frac{1}{2 \pi} \, \int_0^{2 \pi} \, \ln ((1 + R \cos \theta)^2 + R^2 \sin^2
\theta) \, d \theta
= \frac{1}{2 \pi} \, \int_0^{2 \pi} \, \ln (1 + R^2 + 2 R \cos \theta)  \, d
\theta,$$
which is known to be $2 \ln R$ if $R \geq 1$, and is $0$  if $R \leq 1$.    In our case $R = 1$.
That is, $\Delta(1 + x) = e^0 = 1$.

The second statement in (1) follows from the first by setting $x = 1-a$.

For (2), the first claim easily follows from part (1). For the second note that if  $a = \Phi(x)^{-1} x - 1$, then $\Phi(a) = 0$.  
Hence $\Phi(a^m) = 0$, and $\tau(a^m) = \tau(\Phi(a^m)) = 0$,  for all $m \in \Ndb$.  If $a$ is even a contraction, then $r(a)\leq 1$, and so by (1) we then have 
$$1=\Delta(1 + a) = \Delta( \Phi(x)^{-1} x) = \Delta( \Phi(x))^{-1} \Delta( x)   = \Delta( \Phi(x))^{-1} \Delta( x).$$
Hence $\Delta(\Phi(x)) = \Delta(x).$

Next we prove (1) and (2)  in the case that $M$ is a von Neumann algebra with a possibly nonfaithful normal tracial state $\tau$.
We use the notation above Lemma \ref{exucexp}, so that
 $N =  zM, \tau' = \tau_{|N}$.  Let $\Delta = \Delta_{\tau'}$ be the
Fuglede-Kadison determinant for $N$.  
   If  $\tau(x^m) = 0$ for all $m \in \Ndb$ then $0 = \tau'(z x^m) = \tau'((zx)^m)$ for all $m \in \Ndb$.
Clearly $r(xz) \leq 1$ in $zM$ by the proof of Lemma \ref{zcalc}.
 Thus by   (1) in the faithful tracial state case we obtain
$\Delta_\tau(1 + x)  = \Delta(z(1+x)) =\Delta(1_N + zx) = 1$.  The other part of (1) is as in the faithful tracial state case.

For (2), that $x \in {\rm Ker}(\Phi)$ again implies that
 $\tau(x^n) = \tau(\Phi(x^n) ) = 0 = \tau'((zx)^n)$
 for all $n \in \Ndb$.    Thus by  (1) we obtain
$\Delta_\tau(1 + x)  = \Delta(z(1+x)) =\Delta(1_N + zx) = 1$.    The proof of the last statement in (2)
 is just as in the faithful tracial state case, with $\Delta_\tau$ in place of $\Delta$.

Finally, the general case.
As observed before Lemma \ref{FK}, $\tau^{**}$ is a normal tracial state on $M= B^{**}$.  For (1), since             
 $\tau^{**}(x^m) = 0$ and the spectral radius of $x$ is not increased in $M$,  by  (1) in the
 the von Neumann algebra case we have $\Delta_\tau(1 + x) \geq 1$.  The second assertion of (1) is as before.

For (2), if $x \in {\rm Ker}(\Phi)$ then again 
$\tau(x^m) = \tau(\Phi(x^m)) = 0$.    Thus by  (1) we obtain
$\Delta_\tau(1 + x) = 1$.    The proof of the last statement in (2)
 is just as for the analogous statement in  the von Neumann algebra case
above.   \end{proof}

{\bf Remark.} If $\Phi : A \to D$ is as in (2), and if  $M$ is a von Neumann algebra and $D$ and $A$ are weak* closed,
and if $\tau$ is faithful, 
then $A$ is a tracial subalgebra of $M$ in the sense of \cite{BLI}.  
If still further $A+ A^*$ is weak* dense in $M$ then $A$ is a  maximal
subdiagonal subalgebra in the sense of Arveson {\rm \cite{Arv}).  

\bigskip

The above gives  a weaker form of the classical Bishop result on the existence of Jensen measures, with a new proof.  Namely
suppose that $\varphi$ is a character on a closed unital subalgebra $A$ of $C(K),$ for a compact space $K$.  Then by the Hahn-Banach theorem
$\varphi$ is the restriction of a tracial state $\tau$ on $C(K).$   If $\mu$ is the probability measure on $K$ corresponding to $\tau$ then
by a remark in the Introduction, $\Delta_\tau(f)
 = \exp ( \int_K \, \ln |f| \, d \mu),$ for $f \in C(K)$.  By  Theorem \ref{baly},
$\exp ( \int_K \, \ln |1 +f| \, d \mu ) = 1$ if $f \in {\rm Ball}(A)$ with $\varphi(f) = 0$.   Also,
if $f \in A$  with  $\varphi(f) \neq 0$, and if the range of $f$ lies in the closed disk
centered at $\varphi(f)$ with radius $|\varphi(f) |$, then  $\exp ( \int_K \, \ln |f| \, d \mu) = |\varphi(f) |$.
  This $\mu$  need not be the desired Jensen measure in Bishop's theorem, indeed the proof above shows that any representing measure for
$\varphi$  satisfies the above.    (See also
 Corollary \ref{JenCd} below.)

It is an interesting open problem as to what conditions (if any) are needed to get  the   existence of Jensen measures
in full generality: for a $D$-character $\Phi : A \to D$, does there exist a tracial state $\psi$ on the containing $C^*$-algebra $B$ such that
$\Delta_\psi(\Phi(a)) \leq \Delta_\psi(a)$ for all $a \in A$.
Note that this forces (as in \cite{Arv}) 
$\Delta_\psi(\Phi(a)) =  \Delta_\psi(a)$ for all $a \in A^{-1}$.   

 \begin{corollary} \label{JenCd} Let $B$ be a $C^*$-algebra (resp.\ von Neumann algebra) with a  faithful (resp.\ faithful normal) tracial state $\tau$.
 Suppose that  $x \in B$, and let $\Gamma$ be the set of (not identically zero) scalar 
homomorphisms
on the closed algebra generated by $x$ and $\I$.  Suppose also that $\tau$ restricts
to a homomorphism on the latter algebra (that is,
$\tau$ is a `noncommuting representing measure' for
 this homomorphism).    If the spectrum of $x$ is contained in the disk centered at $\tau(x)$
of radius $|\tau(x)|$, then   the Jensen equality $|\tau(x)| =  \Delta_\tau(x)$ holds.  \end{corollary}

\begin{proof}  First assume that $\tau(x) \neq 0$.   The disk condition implies that  
 $| \chi(x) - \tau(x) | \leq |\tau(x)|$ for all
$\chi \in \Gamma$, so that $| \chi(1 - x/ \tau(x) | \leq 1$.   That is,  $r(x/\tau(x) - 1) \leq 1$. 
We work in the $C^*$-algebra case, the  von Neumann case being slightly easier.  
Since $\tau((x/\tau(x))^m) = 1$ for all $m \in \Ndb$, we have by Theorem  \ref{baly} that $\Delta_\tau(x/\tau(x)) = 1$ and so 
$|\tau(x)| = \Delta_\tau(x)$.   
If $\tau(x) = 0$ 
then the condition implies that $r(x) = 0$.   Hence $r(\hat{x} z) = 0$ so that the 
Brown measure of $xz$ is the Dirac mass at $0$.   By for example the last assertions of \cite[Theorem 3.13]{Brown}, we have $\Delta_\tau(x)
= \Delta(xz) = 0 = |\tau(x)|$. \end{proof}

\section{Some Gleason-Whitney type results} \label{GW}

In this section we take an intermission from Jensen inequality considerations, and turn to the exclusively weak* (von Neumann algebraic) situation
of $D$-valued Gleason-Whitney type results.
Such results are about existence  and uniqueness of weak* continuous
representing measures (in their case the algebra is $H^\infty(\Ddb)$).
We remark that just as in the classical case there are natural
criteria which ensure the uniqueness of representing measures.
For example, in \cite[Theorem 4.4]{BLI} it is proved that
if an operator algebra $A$ satisfies
a natural generalization of the classical notion of {\em logmodularity}, then every unital completely contractive
homomorphism $\varphi : A \rightarrow B(H)$ extends uniquely to a completely positive and completely
contractive map $B \rightarrow B(H)$.

The following is a `Hoffman-Rossi/Gleason-Whitney-like' theorem.
  
\begin{corollary} \label{HRf}  Let $\tau$ be a (possibly non faithful) normal tracial state on a  von Neumann algebra 
$M$, and let $D$ be a von Neumann subalgebra of $M$ such that $\tau$ is faithful on $D$. Let $\Phi : A \to D$ be a $\tau$-preserving unital 
map on a subalgebra  $A$ of $M$ containing $D$, which is a $D$-bimodule map (or equivalently is the identity map on $D$).  
Then there is a $\tau$-preserving normal conditional expectation $M \to D$ extending $\Phi.$ \end{corollary}

\begin{proof}  The desired expectation is the map $\Psi$ in Lemma \ref{exucexp}. 
Let $z$ be the support projection for $\tau$ as before.
  The map $\rho : a z\mapsto \Phi(a)z$ from $Az$ to $Dz$ is
well defined and $\tau$-preserving.  This is 
similar to 
Lemma \ref{exucexp} 
(or to the alternative argument that $\rho'$ has those properties there):  
  If $m, m' \in A$ with $m z = m'z$ then 
$$\tau(d \Phi(m-m')) = \tau(d(m-m')) = \tau(d(m-m')z) = 0 , \qquad d \in D.$$
Thus $\Phi(m-m') z = 0$, so that $\rho$ is well defined.   Now $$\tau(\rho(az) ) = \tau(\Phi(a)) = \tau(a)
= \tau(az),$$
so $\rho$ is $\tau$-preserving. 
 It clearly 
is a unital idempotent $Dz$-module map  onto $Dz$.  As in the proof of Theorem \ref{baly}
we see that $Dz = Az \cap A^* z$ and $\rho$ is the restriction to $Az$ of the unique $\tau$-preserving normal conditional expectation $\Psi'$ from $Mz$ to $Dz$.
By the proof of Lemma \ref{exucexp} we see that 
$\Psi(m) z = \Psi'(mz)$, where
 $\Psi$ is the map in Lemma \ref{exucexp}.   So $$\Psi(a) z
 = \Psi'(az) = \rho(az) = \Phi(a) z, \qquad a \in A .$$
 Hence $\Psi$ extends $\Phi$.
\end{proof}
 
{\bf Remarks.} 1) \ 
    The trick above and in the last section involving the reduction by the projection $z$
    will yield a theory of  subdiagonal subalgebras of  von Neumann algebras with 
  a  normal tracial state that is not  necessarily faithful.  
Indeed this theory reduces to the usual theory with respect to the subalgebras $Dz \subset Az \subset Mz$. 
As in the proof of Corollary \ref{HRf} there is 
an appropriate homomorphism $Az \to Dz$.   If $a_t + b_t^* \to x \in M$ weak* then  $a_t z + (b_t z)^* \to xz$ weak*.
Thus the weak* closure of $Az$ is a subdiagonal subalgebra of $Mz$.   

\medskip

2) \  Let $\tau$ be a tracial state on a $C^*$-algebra $M$, and let $J$ be its left kernel as above.   If $\Phi : A \to D$ is a $D$-bimodule map on a subalgebra $A$ of $M$
onto a $C^*$-subalgebra $D$ of $A$, which is $\tau$-preserving, then 
$\Phi(J \cap A) \subset J \cap D$.   Indeed if $x \in J \cap A$ and $d \in D$ then 
$$\tau(d \Phi(x)) = \tau(\Phi(d x)) = \tau(d x) = 0 , \qquad d \in D.$$  
Thus $\Phi(x) \in J \cap D$.  

Thus  if $\Phi$ is also a homomorphism, for example, then it induces a 
completely  contractive homomorphism
$\tilde{\Phi} : A/(J \cap A) \to D/(J \cap D)$.
Now $D/(J \cap D)$ may be identified with $Dz$.   Indeed multiplication by $z$ on $D$ is a $*$-homomomorphism onto $Dz$ with kernel $M z^\perp \cap D = J \cap D$.   One might expect that $\tilde{\Phi}$ is a disguised form 
of the canonical map $Az \to Dz$ considered elsewhere in this section,
but this is in general not the case.  Indeed $J \cap A$ often equals
$J \cap D$, which is $(0)$ if $\tau$ is faithful on $D$.  In 
such cases $A/(J \cap A) = A$. However $Az$ need not be isomorphic to $A$.
 Similarly, it does not seem feasible to replace our $Dz \subset Az \subset Mz$ arguments above by 
quotient space arguments with
$D/(J \cap D), A/(J \cap A), M/J$. 

\begin{corollary} \label{HRab}  \ Let $M$
 be a commutative von Neumann algebra, 
and let $D$ be a $\sigma$-finite von Neumann subalgebra of $M$.   Let $\Phi : A \to D$ be a unital weak* continuous linear 
map on a weak* closed subalgebra  $A$ of $M$ containing $D$, which is a $D$-bimodule map (or equivalently is the identity map on $D$).   Then there is a 
 normal conditional expectation of $M$ onto $D$ extending $\Phi$, if and only if for some (or for every)  faithful normal state $\tau$ on $D$ there exists 
a normal state  $\psi_\tau$ on $M$ extending $\tau \circ \Phi$.  \end{corollary}

\begin{proof} If  $\Psi : M \to D$ is a  normal conditional expectation extending $\Phi$ and $\tau$ is a  faithful normal state on $D$
then $\psi_\tau = \tau \circ \Psi$  is
a normal state $\tau'$ on $M$ extending $\tau \circ \Phi$.    

Conversely, suppose that there is
a normal state $\psi_\tau$ on $M$ extending $\tau \circ \Phi$.   Then the conditions of Corollary \ref{HRf} are met with $\tau$ replaced by $\psi_\tau$, so that
 there exists a $\psi_\tau$-preserving normal conditional expectation $M \to D$ extending $\Phi.$ 
\end{proof}

We pause to give two $D$-valued Gleason-Whitney type results of a similar flavor to the last couple of results:

\begin{proposition} \label{GWh}  Suppose that $A$ is a weak* closed
unital  subalgebra of a von Neumann algebra $M$. 
Suppose that $\Phi : A \to D$  is a unital weak* continuous 
 contractive  linear map into a  von Neumann algebra $D$.  If $A$ has the 
Gleason-Whitney type property that for every weak* continuous state $\varphi$ 
on $D$, $\varphi \circ \Phi$  has a unique  normal state extension  to $M$,
then $\Phi$ has a unique normal positive extension $\Psi : M \to D$.    \end{proposition} 

\begin{proof}  For a weak* continuous state $\varphi$ on $D$
we have that $\varphi \circ \Phi$ has a unique  normal state extension 
$\psi_\varphi$ to $M$,
by hypothesis.
It is easy to see that this property holds with `state' replaced by
positive multiple of a state.  The map
$\varphi \to \psi_\varphi$ preserves convex combinations, and norms.   Indeed for $a \in A, t \in [0,1],$
we have that
$$\psi_{t \varphi_1 + (1-t) \varphi_2}(a) = (t \varphi_1 + (1-t) \varphi_2)(\Phi(a)) = (t \psi_{\varphi_1} + (1-t) \psi_{\varphi_2})(a),$$
where $\varphi_1,  \varphi_2$ are normal states or positive multiples of such.  
By uniqueness
of the extension, $\psi_{t \varphi_1 + (1-t) \varphi_2} = t \psi_{\varphi_1} + (1-t) \psi_{\varphi_2}$.
Indeed we have that $\psi_{\varphi_1 +  \varphi_2} = \psi_{\varphi_1} +  \psi_{\varphi_2}$ and
$\psi_{t \varphi_1} = t \psi_{\varphi_1}$ if $\varphi_1,  \varphi_2$ are positive normal  functionals on $D$ and $t \geq 0$.
It then follows easily that for such  $\varphi_1,  \varphi_2$ the map $\varphi_1 - \varphi_2 \mapsto  \psi_{\varphi_1} -  \psi_{\varphi_2}$ is well defined,
and its domain is the selfadjoint normal  functionals on $D$.   If $\varphi_k$ is a  positive normal  functional on $D$ for 
$k = 1, 2, 3, 4$, then we obtain that the map 
$$T : \varphi_1 - \varphi_2 + i(\varphi_3 - \varphi_4) \mapsto  \psi_{\varphi_1} -  \psi_{\varphi_2} + i( \psi_{\varphi_3} -  \psi_{\varphi_4}) \in M_*$$
is well defined, and its domain is $D_*$.   If $\varphi_k$ is also contractive for all $k$ we see that 
$\| T(\varphi_1 - \varphi_2 + i(\varphi_3 - \varphi_4)) \| \leq 4$.   Thus by the Hahn-Jordan decomposition in $D_*$ it follows that 
$T$ is bounded (we will see momentarily that it is contractive).   For $a \in A$ and a normal state $\varphi$ on $D$ we have
$T^*(a) (\varphi) = T(\varphi)(a) =\Phi(a)( \varphi)$.   It follows that $\Psi = T^*$ extends $\Phi$.  In particular $T^*(1) =  1$.   For $x \in M_+$ we have 
$T^*(x)(\varphi) = \psi_\varphi(x) \geq 0$.   Thus $T$ is positive.
It is well known that a unital map between $C^*$-algebras  is positive  if and only if it is contractive. 
For the uniqueness note that if $D \subset B(K)$, then  for every unit vector $\xi \in K$, we have that 
$\langle \Psi(\cdot) \xi , \xi \rangle$ is the unique normal state extending $\langle \Phi(\cdot) \xi , \xi \rangle$.  \end{proof}

Note that by \cite[Theorem 4.1]{BL-FMR} or  \cite[Lemma 5.8]{BLueda}, the 
Gleason-Whitney type property that  every weak* continuous state
on $A$ has at most one normal state extension  to $M$, 
is equivalent to  $A + A^*$ being weak* dense in $M$.   In this case it is evident that 
any $\Phi : A \to B(H)$ has at most one weak* continuous extension to $M$.

\begin{proposition} \label{GWh2}  Suppose that $A$ is a weak* closed
unital  subalgebra of a von Neumann algebra $M$.   
Suppose that $\Phi : A \to D \subset B(H)$
 is a unital weak* continuous 
completely contractive  linear map into a  von Neumann algebra $D$ on $H$.  If $A$ has the 
Gleason-Whitney type property that  every state extension to $M$ of
a weak* continuous state on $A$ is normal, then $\Phi$ has a normal completely positive
extension $M \to B(H)$.    \end{proposition} 

 \begin{proof}  
Let $\Psi : M \to B(H)$ be a completely contractive linear extension of $\Phi$. 
The existence of such an extension follows from Arveson's extension of the Hahn-Banach theorem to $B(H)$-valued maps (see \cite[Theorem 1.2.9]{Arv2}).
For each $\zeta \in H$ of norm $1$, $\langle \Psi(\cdot) \zeta , \zeta \rangle$ is a state
extension of its restriction  $\langle \Phi(\cdot) \zeta , \zeta \rangle$ to $A$.   By the Gleason-Whitney hypothesis 
 $\langle \Psi(\cdot) \zeta , \zeta \rangle$  is weak* continuous on $M$, and by scaling
this holds for all $\zeta \in H$.  By polarization,   $\langle \Psi(\cdot) \zeta , \eta \rangle$  is weak* continuous on $M$
 for all $\zeta, \eta \in H$.    If $x_t \to x$ is a bounded weak* converging net in $M$ then
$\Psi(x_t) \to \Psi(x)$ WOT.  It follows that  $\Psi$ is weak* continuous.
\end{proof}

\section{The Bishop-Ito-Schreiber theorem and the characterization of homomorphisms}  \label{BIS}

The classical  Bishop-Ito-Schreiber theorem  states 
that the existence of a Jensen measure actually  characterizes the characters among the linear functionals
on a uniform algebra.  
Here we present noncommutative variants of this result.
One complication that occurs in the noncommutative case requires us
to first establish some results about Jordan homomorphisms.
 
\begin{lemma} \label{l1}
Let $B$ be a subalgebra of an algebra $A$. If $\Phi : A \to B$ is an idempotent 
$B$-module map  whose kernel is a  subalgebra of $A$, 
 then $\Phi$ is a 
homomorphism.   In particular this holds if ${\rm Ker}(\Phi)$ is a left or right ideal in $A$.  \end{lemma} \begin{proof}  For $x, y \in A$ we have 
$$\Phi(xy) = \Phi((x - \Phi(x))(y - \Phi(y)) +  \Phi(x)y + x \Phi(y) - \Phi(x) \Phi(y)) = \Phi(x) \Phi(y) ,$$
since $\Phi(\Phi(x)y) = \Phi(x \Phi(y)) = \Phi(x) \Phi(y)$ and $\Phi(\Phi(x) \Phi(y)) = \Phi(\Phi(\Phi(x)y)) =  \Phi(x) \Phi(y)$.
\end{proof}  

\begin{lemma} \label{l2}
For any Jordan homomorphism $\Phi : A \to B$ between algebras,  and $a \in {\rm Ker} ( \Phi )$ and $b \in A$, we have that
$\Phi(ab)$ and $\Phi(ba)$ have square zero.     \end{lemma}  \begin{proof} 
We use an idea from \cite{Zel}.
We have $\Phi(ab) + \Phi(ba) = \Phi(a) \Phi(b) + \Phi(b) \Phi(a) = 0$,
so that $\Phi(ab) = -\Phi(ba)$.  Defining $x = \Phi(ab)$ we have
$$4 x^2 = \Phi((ab-ba)^2) = x^2 + x^2 - \Phi(ab^2a) - \Phi(b a^2b) = 2x^2.$$
We have here used the identity  $\Phi(aba)=\Phi(a)\Phi(b)\Phi(a)$ which follows from e.g.\ a formula on 
p.\ 208 in Section 3.2 of \cite{BratteliRobinsonV1}, or 
from the identity $aba = 2 (a \circ b) \circ a -  a^2 \circ b$, where $\circ$ is the Jordan product.
So $x^2 = 0$.  \end{proof}

\begin{corollary} Suppose that $B$ is a closed selfadjoint subalgebra
of an operator algebra $A$, and that $P : A \to B$ is an idempotent  Jordan homomorphism and $B$-bimodule map.
If $B$ has no nonzero square zero elements 
then $P$ is a homomorphism.  \end{corollary} 

\begin{proof}   Under these hypotheses it follows from Lemma \ref{l2} that Ker$(P)$ is an ideal, and then from Lemma \ref{l1} that
$P$ is a homomorphism.  \end{proof}

\begin{corollary} \label{lfix} Suppose that $B$ is a closed selfadjoint subalgebra
of an operator algebra $A$, and that $P : A \to B$ is an idempotent  Jordan homomorphism and $B$-bimodule map.
Then $P$ is a homomorphism.  \end{corollary} 

\begin{proof} Let $a \in {\rm Ker} ( \Phi )$ and $b \in A$ be given. Select $d\in B$ so that $d\Phi(ab)=|\Phi(ab)|$. It is clear that then $\Phi(da)=d\Phi(a)=0$ and $\Phi(dab)=|\Phi(ab)|$. By  Lemma \ref{l2} we have $|\Phi(ab)|^2=\Phi(dab)^2 = 0$, and hence $\Phi(ab)=0$. By Lemma
\ref{l1}, $P$ is a homomorphism.
 \end{proof}

{\bf Example.} We construct   an example of a completely contractive
unital Jordan homomorphism and projection $P : A \to B$, from a unital operator algebra $A$ onto its closed subalgebra $B$,
 which is a $B$-bimodule map, but is not an algebra homomorphism.
This shows the importance of $B$ being  selfadjoint in Corollary
\ref{lfix}.

Consider the set $J$
of
$4 \times 4$ matrices
$$ \left[ \begin{array}{ccccl}
0  & \alpha & \beta & 0 \\ 0 & 0  & 0  & - \beta \\
0 & 0  & 0  &  \alpha \\
0 & 0  & 0  & 0
\end{array}  \right] , \qquad \alpha, \beta  \in \Cdb. $$
This is not an associative algebra but is a Jordan operator algebra with zero Jordan 
product.
We let $A = \Cdb I_4 + J + \Cdb E_{14}$, and $B = \Cdb I_4 + \Cdb E_{14}$.
It is easy to see that $A$ is a subalgebra of $M_4$ with subalgebra $B$.  
Let $P(\alpha I_4 + x + \beta E_{14}) = \alpha I_4 +\beta E_{14}$
for $\alpha, \beta \in \Cdb, x \in J$.  Clearly $P$ is unital. For such $\alpha, \beta, x$ let
$a = \alpha I_4 + x + \beta E_{14}$.  Then   $$P(a^2) =
P(\alpha^2 I_4 + 2 \alpha \beta E_{14} + y) = \alpha^2 I_4 + 2 \alpha \beta E_{14}
= P(a)^2,$$
for some $y \in J$.  So $P$ is a Jordan homomorphism.
 On the other hand $$P((E_{12} + E_{34})(E_{13} - E_{24})) =
-P(E_{14}),$$ whereas $P(E_{12} + E_{34}) = P(E_{13} - E_{24}) = 0$.
So $P$ is not a homomorphism.

To see that $P$ is completely contractive note that
removing the middle two rows from a matrix in $A$, then removing
the middle two columns, is completely contractive.
One is left with a $2 \times 2$ matrix algebra completely
isometrically isomorphic to $B$.   The composition of these
procedures equals $P$, so $P$ is completely contractive.
Finally, that $P$ is a $B$-bimodule map follows either from a simple
direct computation, or from Proposition 5.1 of \cite{BLI}.

\bigskip

The following is  a   
$C^*$-algebraic variant of the Bishop-Ito-Schreiber theorem.
 The classical Bishop-Ito-Schreiber theorem is stated
  in the introduction.
 
\begin{theorem} \label{bis} Suppose that $\tau$ is a normal tracial state on
a $C^*$-algebra $M$, that $D$ is a unital $C^*$-subalgebra of $M$, that $\tau$ is faithful on $D$, that $A$ is a
unital subalgebra of $M$ containing $D$, and that $\Phi : A \to D$ is a unital  $D$-bimodule map. Then   $\Phi$  
satisfies the ball-Jensen inequality  if and only if $\Phi$ is a $\tau$-preserving
homomorphism.   
If these hold and $M$ is a von Neumann algebra, $\tau$ is normal,
and if $D$ is weak* closed, then $\Phi$
is the restriction to $A$ of the canonical $\tau$-preserving faithful conditional expectation of
$M$ onto $D$.  If in addition 
$\tau$ is faithful on $M$ then we have $D = A \cap A^*$.
\end{theorem} 

\begin{proof}  
Theorem \ref{baly} gives the `if' part of the first equivalence (and this does not need  $\tau$ to be faithful on $D$).
For the other direction, we modify an idea from \cite{Sid}.  
Write $\Delta$ for $\Delta_\tau$.
 Suppose that $\Phi$  satisfies
the ball-Jensen inequality.
If $f \in {\rm Ker} (\Phi)$ and $z \in \Cdb$ then since $\Delta(\cdot) \leq  \tau(|\cdot|)$ (see Lemma \ref{FK}), we have
$$1 = \Delta(\Phi(1 - z f))
\leq \Delta(1 - z f) \leq \tau(|1-zf|) \leq
\tau(|1 - z f|^2)^{\frac{1}{2}} ,$$  provided $|z|$ is small enough.  Squaring,
we see that
$2 {\rm Re} \, z \tau  ( f) \leq |z|^2 \tau(|f|^2)$ for all $z \in \Cdb$ with
 $|z|$  small enough.   It follows that  $\tau  ( f) = 0$.
Hence for any $a \in A$ we have  $\tau(a) = \tau(a - \Phi(a)) + \tau(\Phi(a)) =
\tau(\Phi(a))$.  That is,
$\Phi$ is $\tau$-preserving.    

Similarly, we also have
$$1  \leq  \Delta(1 - z f) \,  \Delta(1 + z f) 
\leq \Delta((1 - z df)  (1 + z df)) = \Delta(1 - z^2 (d f)^2)$$  when  $|z|$ is small enough.
Hence $1 \leq \tau(|1 - z^2 (d f)^2|) \leq \tau(|1 - z^2 (d f)^2|^2)^{\frac{1}{2}}.$
Squaring, and replacing $z^2$ by $w \in \Cdb$ we see that
$2 {\rm Re} \, \tau  (w (d f)^2) \leq |w|^2 \tau(|(df)^2|^2)$ for all $w \in
\Cdb$  with
 $|w|$  small enough.   It follows easily
that $\tau((df)^2) = 0$.     Thus $$0 = \tau((d+1) f (d+1)f - (d-1) f (d-1)f) =
2\tau(df^2 + f df) = 4 \tau(df^2) .$$
Hence $\tau(D f^2) = \tau(D \Phi(f^2)) = 0,$ so
that $\Phi(f^2) = 0$ since $\tau$ is faithful on $D$.  

It follows as in the proof of  Lemma \ref{l1}  that if $f \in A$ then
$$\Phi(f^2) = \Phi(((f - \Phi(f)) +\Phi( f))^2) = \Phi((f - \Phi(f))^2) + 0 + \Phi(\Phi(f)^2) =  \Phi(f)^2.$$
Hence $\Phi$ is a Jordan homomorphism.   By Corollary \ref{lfix} it is a homomorphism.

For the final assertion,
by Lemma \ref{exucexp} and  \cite[Lemma 5.3]{BLI}, 
$\Phi$  is the restriction of the unique $\tau$-preserving faithful conditional expectation of $M$ onto $D$. 

As in the 
first paragraph of the 
proof of Theorem \ref{baly}, if $\tau$ is faithful 
on $M$ then  $D = A \cap A^*$.
\end{proof}  
 
{\bf Remark.}  We did not use the full strength of the ball-Jensen inequality in the last proof, and indeed the proof works
if $A$ merely satisfies the requirement that there exists $\delta > 0$ such that $\Delta(1 + x) \geq 1$ for 
all $x \in A$ with $\Phi(x) = 0$ and $\| x \| < \delta$. 
We could alternatively use  the spectral radius here in place of $\| x \|$.

\begin{theorem} Suppose that $\tau$ is a tracial state on a $C^*$-algebra $M$, that $D$ is a  
$C^*$-subalgebra
of $M$, and that $\tau$ is faithful on $D$.
Let $A$ be a unital subalgebra of $M$ containing $D$, and  $\Phi : A \to D$ a unital $\tau$-preserving (that is, $\tau \circ \Phi = \tau$) contraction and $D$-bimodule map. 
The map $\Phi$ is a 
 homomorphism   if and only if for all quadratic polynomials $p(s,t)$ of two variables, $\| p(d, \Phi(a) ) \|_2 \leq \| p(d,a) \|_2$ for $d \in D, a \in A$.
 \end{theorem}

\begin{proof}  
First let $\Phi$ be a Jordan homomorphism. We know that $\Phi$ is 2-contractive on $A$. Indeed this follows from the relation
$$\tau(\Phi(a)^* \Phi(a)) = \tau(\Phi(\Phi(a)^* a)) = \tau(\Phi(a)^* a) \leq ||\Phi(a)||_2 \, ||a||_2, \qquad a \in A.$$
Given a quadratic polynomial $p$ of two variables, the inequality $\| p(d, \Phi(a) ) \|_2= \| \Phi(p(d,a)) \|_2 \leq \| p(d,a) \|_2$ for $d \in D, a \in A_0 = A \cap \, {\rm Ker} (\Phi)$, will therefore follow once we are able to prove that $p(d, \Phi(a) ) = \Phi(p(d,a))$. Let $d \in D, a \in A_0 = A \cap \, {\rm Ker} (\Phi)$ be given. Using the fact that $\Phi$ is both a $D$-bimodule map and a Jordan homomorphism, it is easy to verify that for terms of the form $d^na^md^k$ (where $n+k+m\leq 2$), we have that $\Phi(d^na^md^k)=d^n\Phi(a)^md^k$. Since also $\Phi(ada)=\Phi(a)\Phi(d)\Phi(a)=\Phi(a)d\Phi(a)$, it follows that $p(d, \Phi(a) ) = \Phi(p(d,a))$ as required.

Conversely, suppose that the inequality holds. If $d \in D_+, a \in A_0 = A \cap \, {\rm Ker} (\Phi)$, then 
$$\tau(d^2) = \tau(|d -z\Phi(a)^2|^2) \leq \tau(|d - z a^2|^2) ,$$
using the inequality in the hypothesis. We conclude that 
$\tau(d^2) \leq \tau(|d - za^2|^2) ,$ or equivalently
$$2 {\rm Re} \, ( \overline{z} \tau(da^2)) \leq \,  |z|^2 \tau(|a^2|^2).$$For an appropriate choice of $z = r e^{i \theta}$ this yields
$$2 r |\tau(da^2)| \leq \,  r^2 \tau(|a^2|^2).$$Using the fact that on $A$ we have $\tau\circ\Phi=\tau$, we may conclude from this 
that $2 |\tau(d\Phi(a^2))| \leq \,  r \tau(|a^2|^2)$. This inequality holds for any $r>0$, which in turn ensures that $\tau(d\Phi(a^2))=0$. Since $d \in D_+$ was arbitrary, we have that $\tau(d\Phi(a^2))=0$ for all $d \in D$, and hence that $\Phi(a^2)=0$. Next it follows  as in the proof of  Lemma \ref{l1}  that $$\Phi((d + a)^2) = \Phi (a)^2 + \Phi(ad + da) + d^2 = d^2
= \Phi(d+a)^2$$
for $d \in D$. So $\Phi$ is a Jordan homomorphism on $A$.   By Corollary \ref{lfix} it is a homomorphism.
\end{proof}  

{\bf Remark.}  There is a simple proof that  the map $\Phi$ in the last result is a homomorphism if and only if for all polynomials $p(s,t,v)$ of three variables, $$\| p(d, \Phi(a), \Phi(b) ) \|_2 \leq \| p(d,a,b) \|_2 , \qquad d \in D, a,b \in A.$$
Indeed, if $\Phi$ is a homomorphism, it easily follows from the 2-contractivity of $\Phi$ on $A$ that  $\| p(d, \Phi(a),\Phi(b) ) \|_2= \| \Phi(p(d,a,b)) \|_2 \leq \| p(d,a,b) \|_2$ for $d \in D, a,b \in A$.   We had included a proof of the converse in a previous 
version of our paper.

\section{Gleason parts} \label{gp} 

We note that the Gleason relation $\| \Psi - \Phi \| < 2$
does not seem to have a $B(H)$-valued analogue. 
To see this  set
$A = {\mathcal U}(E)$ for an operator space $E \subset B(H)$, that is the upper triangular matrices
with $1$-$1$ and $2$-$2$ entries scalar multiples of $I_H$, and the $1$-$2$ entry in $E$.   Consider 
 completely contractive unital homomorphisms  $\Psi_T$ on $A$ 
induced by a linear 
complete contraction $T: E \to B(K)$ as in e.g.\ \cite[Proposition 2.2.11]{BLM}.   The Gleason relation $\| \Psi - \Phi \| < 2$ is easily seen not 
to be an equivalence relation, since the analogous relation on linear complete contractions from $E$ to $B(K)$ is not. 

To show that the Gleason relation 
does work for $D$-characters,
we will use some concepts considered by Harris in e.g.\ \cite{Harris0,Harris}.
Write $T_{x}(y) = (1-xx^*)^{-\frac{1}{2}} (x+y) (1+x^*y)^{-1} 
(1-x^* x)^{\frac{1}{2}}$.    This makes sense for strict contractions $x, y$ on a Hilbert space $H$, 
that is for elements in the  open unit ball in $B(H)$.
 For a fixed strict contraction $x$ on $H$ the maps $T_x$ are essentially exactly the 
biholomorphic self maps of the  open unit ball in $B(H)$ (see e.g.\ \cite[Theorem 3]{Harris0}), and we call these 
{\em M\"obius maps} of
the open ball.   The {\em 
hyperbolic distance} is  $$\rho(x,y) = \tanh^{-1} \| (1-xx^*)^{-\frac{1}{2}} (x-y) (1-x^*y)^{-1} 
(1-x^* x)^{\frac{1}{2}} \| =  \tanh^{-1} \| T_{-x}(y) \|.$$   This is a metric on the  strict contractions on $H$.
Harris shows (see p.\ 356 and Exercise 6 on
p.\ 394 of \cite{HarHol}) that 
$\rho$ is what is known as a CRF pseudometric
on the open unit ball ${\mathcal U}_0$ in any $J^*$-algebra
and it satisfies the Schwarz-Pick inequality 
$$\rho(h(x),h(y)) \leq \rho(x,y), \qquad x, y \in {\mathcal U}_0,$$
for any holomorphic $h : {\mathcal U}_0 \to {\mathcal U}_0$. 
We pause to remind the reader that a map $h:V\to E_2$ from an open subset of some complex locally convex space $E_1$ into another complex locally convex space $E_2$, is deemed to be holomorphic on $V$ if at each $x\in V$, the Fr\'echet derivative of $h$ exists as a continuous complex linear map from $E_1$ to $E_2$. There is a similar result
for holomorphic maps $h$ between the open unit balls of two $J^*$-algebras.  
We have equality in  the Schwarz-Pick inequality  if $h$ is biholomorphic (onto
${\mathcal U}_0$) of course.
  We do not need all of the following facts, but state them since they do not seem to be in the literature.  

\begin{lemma} \label{Harrt}   Let $S_n, T_n$ be strict contractions on a Hilbert space
$H$.   If $\| S_n - T_n \| \to 2$, or if there is a constant $c < 1$ with
$\| S_n \| \leq c$ and  $\| T_n \| \to 1$, then $\rho(S_n, T_n) \to \infty.$ 
Also, for  strict contractions $x, y \in B(H)$ we have $$1 - \| T_{x}(y) \|^2 \leq \frac{1+ \| x \|^2}{(1-\| x \| \| y \|)^2} \; (1- \| y \|)^2  .$$
\end{lemma}

\begin{proof}    If $\| S_n - T_n \| \to 2$ choose $\varphi_n \in {\rm Ball}(B(H)^*)$ with 
$\varphi_n(S_n - T_n) \to 2$.     So $\varphi_n(S_n) - \varphi_n(T_n) \to 2$.  By the geometry of the disk, 
the (scalar) hyperbolic distance $\rho_{\Ddb}(\varphi_n(S_n), \varphi_n(T_n) ) \to \infty .$   Observe that the restriction of each $\varphi_n$ to the open unit ball ${\mathcal U}_0$
of $B(H)$ is holomorphic.   So by Harris' Schwarz-Pick inequality above, we have 
$$\rho(S_n, T_n) \geq \rho_{\Ddb}(\varphi_n(S_n), \varphi_n(T_n) ) \to \infty .$$  
A similar argument proves the second case, but now choosing $\varphi_n \in {\rm Ball}(B(H)^*)$ with 
$\varphi_n(T_n) \to 1$.   
Since $|\varphi(S_n)| \leq c$, by the geometry of the disk we see again that
$\rho(S_n, T_n) \geq \rho_{\Ddb}(\varphi_n(S_n), \varphi_n(T_n) ) \to \infty .$  
Alternatively, the second 
 follows from the inequality 
$0 < 1 - \tanh \rho(-x,y) = 1 - \| T_{x}(y) \|^2 \leq \frac{1+ \| x \|^2}{(1-\| x \| \| y \|)^2} \; (1- \| y \|)^2)$ above.
To prove this inequality for $\| x \| < 1 , \| y \| < 1,$ we first observe that
by the functional calculus for $|y|^2$ and the $C^*$-identity, we have
\begin{equation}\label{eq:fcalc}
\| (1 - |y|^2)^{-1} \| = (1 - \| y \|^2)^{-1}.
\end{equation}
Let $w = (1-x^* x)^{-\frac{1}{2}}
(1+x^*y)$.
An algebraic identity attributed on p.\ 10 of
\cite{Harris} to \cite[Chapter 1, Section 1]{Pot} (the proof of which also appears  on p.\  78 in Harris' thesis)
states that
$$1 - |T_{x}(y)|^2 = (w^*)^{-1} \,
(1 - |y|^2) \, w^{-1} .$$
Hence $1 - |y|^2 = w^* \, (1 - |T_{x}(y)|^2) \, w$, so that
$$(1 - \| y \|^2)^{-1} = \| (1 - |y|^2)^{-1} \|
= \| w^{-1} \, (1 - |T_{x}(y)|^2)^{-1} \, (w^*)^{-1} \|.$$
Using the fact that the norm equals the spectral radius
on positive elements, and the well known
identity $r(xy) = r(yx) \leq \| y \| \| x \|$ for
the spectral radius,
we see that the last quantity equals
$$r((w w^*)^{-1} \, (1 - |T_{x}(y)|^2)^{-1})
\leq \| (w w^*)^{-1} \| \, \| (1 - |T_{x}(y)|^2)^{-1} \|
= \| (w w^*)^{-1} \| \, (1 - \| T_{x}(y) \|^2)^{-1} . $$
The last equality follows by the norm identity noted in equation (\ref{eq:fcalc}) above.
Thus we will be done if $\| (w w^*)^{-1} \|
\leq \frac{1+ \| x \|^2}{(1-\| x \| \| y \|)^2}$.
To see this, first write
$$(w w^*)^{-1} = (1-x^* x)^{\frac{1}{2}} \,
|1+x^*y|^{-2} (1-x^* x)^{\frac{1}{2}}.$$
Again by the fact that the norm equals the spectral radius
on positive elements, and
the identity $r(xy) = r(yx) \leq \| y \| \| x \|$ above,
we obtain that $\| (w w^*)^{-1} \|$ equals
$$r(|1 +x^*y|^{-2} \, (1-x^* x))
\leq \| |1 +x^*y|^{-2} \| \, \| 1-x^* x \|
= \| (1 + x^*y)^{-1} \|^2 \, \| 1-x^* x \| ,$$
the last equality following by the $C^*$-identity.
By a well known inequality
associated with the Neumann lemma, the last quantity is dominated by $(1 - \| x^*y \| )^{-2} (1 + \| x \|^2) \leq
\frac{1+ \| x \|^2}{(1-\| x \| \| y \|)^2}$ as desired. 
 \end{proof}

We recall that a map $T$ is {\em real positive} if $T(a) + T(a)^* \geq 0$ whenever $a + a^* \geq 0$.

\begin{lemma} \label{newl} Suppose that  $A$ is a closed subalgebra of a $C^*$-algebra $B$, and that $A$ has a contractive approximate identity.   Let $D$ be a $C^*$-subalgebra of $A$, and suppose also that $D$ is identified with 
a $*$-subalgebra of $B(H)$ for some Hilbert space $H$.
Then a linear $D$-bimodule map $u: A \to D$ is real positive if and only
if  $u$ extends to a positive map from  $B$ to $B(H)$.  
\end{lemma} 

\begin{proof}  If $u$ is real positive then  $u$ is bounded 
by e.g.\ \cite[Corollary 4.9]{BW}.   Also, if $[a_{ij}] + [a_{ji}^*] \geq 0$, 
and $d_1, \cdots , d_n \in D$, set $z = \sum_{i,j} \, d_i^* \, a_{ij} \, d_j$.
Then $z + z^* = \sum_{i,j} \, d_i^* \, (a_{ij} + a_{ji}^*) \, d_j \geq 0$.   Thus 
$$\sum_{i,j} \, d_i^* (u(a_{ij}) + u(a_{ji})^*)d_j 
= u(z) + u(z)^* \geq 0 .$$
Hence $[u(a_{ij})] + [u(a_{ji})^*] \geq 0$ by e.g.\ \cite[Lemma 3.2]{Tak1}.
Thus $u$  is real completely positive and extends  by e.g.\ \cite[Theorem 4.11]{BW} to a
completely positive map on $B$.      
Conversely, it is well known (see e.g.\ the lines before \cite[Corollary 4.9]{BW})
that restrictions of  positive maps to $A$ are real positive.
\end{proof}

In the last result for the sake of a cleaner sounding result,
we are abusively identifying $D$ with a $C^*$-subalgebra of both  $B$ and $B(H)$, even though
the latter two algebras may be unrelated.  The same abuse is present in (5) of the next result.
We trust that the benefit of a cleaner result will assuage any offense this abuse may have caused.

\bigskip 

{\bf Remarks.}   Note that the last extension can be done keeping the same norm, 
by an inspection of the results referenced in the proof.

Moreover, the completely positive map from $B$ to $B(H)$
obtained in the last proof, is also a $D$-bimodule map.   This follows
from the following fact: Suppose that $\rho : D \to B$ is a
$*$-homomorphism
between $C^*$-algebras, that $\theta : D \to B(H)$ is a
$*$-homomorphism, and that $T : B \to B(H)$ is a completely positive  map. 
If the restriction of $T$ to $\rho(D)$ is a $D$-bimodule map
(which is equivalent to saying that 
$T(\rho(d)) = \theta(d) T(1)$ for $d \in D$), then 
$T$ is a $D$-bimodule map (that is, $T(\rho(d_1) x \rho(d_2)) =
 \theta(d_1) T(x) \theta(d_2)$ for $d_1, d_2 \in D, x \in B$).
The proof of this just as for Exercise 4.3  in \cite{Pnbook}.

\begin{theorem} \label{Gleq} 
Let $A$ be a unital operator algebra containing
 a  $C^*$-algebra $D$ unitally (i.e.,  with common identity). Let   $\Phi , \Psi : A \to D$ be  $D$-characters.   
 The following 
are equivalent: \begin{enumerate} \item  [{\rm (1)}]
$\| \Phi - \Psi \| < 2$.
\item  [{\rm (2)}] $\| \Phi_{|{\rm Ker} \, \Psi} \| < 1$.
\item  [{\rm (3)}]  There is a constant $M > 0$ with $\rho(\Phi(a), \Psi(a)) \leq M$ for $\| a \| < 1, a \in A$.
\item   [{\rm (4)}] If $\| \Phi(a_n) \| \to 1$ for a sequence $(a_n)$ in ${\rm Ball}(A)$, then 
 $\| \Psi(a_n) \| \to 1$.  \end{enumerate} 
If  $A$ is a subalgebra of a $C^*$-algebra $C$, 
and if $D$ is represented as  a  $C^*$-algebra  nondegenerately on a Hilbert space $H$, then the above conditions are implied by:
 \begin{enumerate} \item [{\rm (5)}]  There are positive constants
$c, d$ and completely positive $B(H)$-valued
maps  $\tilde{\Phi}, \tilde{\Psi}$ extending
$\Phi, \Psi$ to $C$, with $\tilde{\Phi} \leq
c \tilde{\Psi}$ and
$\tilde{\Psi} \leq d \tilde{\Phi}$.  
\item   [{\rm (6)}]   (Harnack inequality) \ There are positive constants
$c, d$ with $\Phi \preccurlyeq c \Psi$ and $\Psi \preccurlyeq d \Phi$ on $A$.
Here $\preccurlyeq$ is the `real positive ordering'; e.g.\  $\Phi \preccurlyeq c \Psi$ means that 
$c \Psi - \Phi$ is a real positive map on $A$.  \end{enumerate}
If $D$ is 1 dimensional then 
{\rm (5)} and {\rm (6)}  are equivalent to the other conditions.   
\end{theorem}

\begin{proof}   We will use the fact
that $D$-characters are  $D$-module maps, 
as we said above.   Several of the arguments below are modifications of the analogous classical proofs.  

 (2) $\Rightarrow$ (1) \ If $\| \Phi_{|{\rm Ker} \, \Psi} \| = c <  1$, and $g \in {\rm Ball}(A)$ set $f = g - \Psi(g)$.   Then 
$\Psi(f) = 0$ so that $\| \Phi(g) - \Psi(g) \| = \| \Phi(f) \| \leq c \| f \| \leq 2 c$.   Thus $\| \Phi - \Psi \| < 2$.

 (1) $\Rightarrow$ (2) \  ($D$ a von Neumann algebra case.) Suppose that  $\| \Phi - \Psi \| = 2c < 2$.  If $f' \in {\rm Ker} \, \Psi, \| f' \| < 1$, choose a unitary $u$  in $D$
such that $u \Phi(f') \geq 0$, and set $f = uf' \in {\rm Ker} \, \Psi$.   Set $g = (c1 - f) (1-cf)^{-1} \in A$.   By the analytic functional 
calculus applied to the M\"obius map $(c1 - z) (1-cz)^{-1}$, we have $\| g \| \leq 1$.   We have
$$ \| \Psi(g) - \Phi(g) \|  = \| c 1 - (c1 - \Phi(f))(1 - c \Phi(f))^{-1} \|_{D} \leq 2c.$$
Since $\Phi(f) = u \Phi(f') \geq 0$, we may regard $\Phi(f)$ as a function $k$ in $C(K)_+$ for $K = \sigma(\Phi(f))$.
We obtain from the last $\leq$ inequality, just as in \cite[Lemma 2.6.1]{Browder}, that $k \leq \frac{2c}{1+c^2}$.   Thus $\| \Phi(f') \| =  \| \Phi(f) \| \leq \frac{2c}{1+c^2}$,
and so $\| \Phi_{|{\rm Ker} \, \Psi} \|   \leq \frac{2c}{1+c^2} < 1$.  

 (1) $\Rightarrow$ (2) \  Suppose that  $\| \Phi - \Psi \| = 2c < 2$.  If $f' \in {\rm Ker} \, \Psi, \| f' \| < 1$, let $x = 
\Phi(f') \in D$.  Choose by
\cite[Proposition 1.4.5]{Ped} (with $a = |x|^2$ there)
a contraction $u$  in $D$ with $x = u |x|^t$ for some $t \in (0,1)$.   We may take $u$ unitary and $t = 1$ if $D$
is a von Neumann algebra.    
In the notation of the proof in that reference $u_n^* x = |x|^{1-t} (|x|^2 + \frac{1}{n})^{-\frac{1}{2}} |x|^2
\to |x|^{2-t}$.     Hence  $u^* x = |x|^{2-t} \geq 0$.  Set $f = u^* f' \in {\rm Ker} \, \Psi$.   Set $g = (c1 - f) (1-cf)^{-1} \in A$.   By the analytic functional 
calculus applied to the M\"obius map $(c1 - z) (1-cz)^{-1}$, we have $\| g \| \leq 1$.   We have
$$ \| \Psi(g) - \Phi(g) \|  = \| c 1 - (c1 - \Phi(f))(1 - c \Phi(f))^{-1} \|_{D} \leq 2c.$$
Since $\Phi(f) = u^* \Phi(f') = |x|^{2-t}  \geq 0$, we may regard $\Phi(f)$ as a function $k$ in $C(K)_+$ for $K = \sigma(\Phi(f))$.
We obtain from the last $\leq$ inequality, just as in \cite[Lemma 2.6.1]{Browder}, that $k \leq \frac{2c}{1+c^2}$.   Thus $\| \Phi(f') \| = \| x \| =  \| \Phi(f) \|^{\frac{1}{2-t}} \leq (\frac{2c}{1+c^2})^{\frac{1}{2-t}}$,
and so $\| \Phi_{|{\rm Ker} \, \Psi} \|   \leq  (\frac{2c}{1+c^2})^{\frac{1}{2-t}} 
 < 1$.  

 (4) $\Rightarrow$ (2) \  If  $\| \Phi_{|{\rm Ker} \, \Psi} \| =  1$ one may contradict (4) by choosing a 
sequence $(a_n)$ in ${\rm Ball}(A) \cap {\rm Ker} \, \Psi$ with $\| \Phi(a_n) \| \to 1$.

 (3) $\Rightarrow$ (2) \  This implication is not needed, but if $\| a \| < 1, a \in A$, and $\Psi(a) = 0$ then by (3) we have
$\rho(\Phi(a), \Psi(a)) = \rho(\Phi(a), 0) = \tanh^{-1} \| \Phi(a) \| \leq M$.   That is, $\| \Phi(a) \| \leq 
\tanh M < 1$.

(2) $\Rightarrow$ (3) \ If (3) were not true then there is a sequence $(a_n)$ in $A$ with $\| a_n \| < 1$ for 
all $n$ and $\rho(\Phi(a_n), \Psi(a_n)) \to \infty$.   Let $b_n = T_{-\Psi(a_n)}(a_n)$.
Since $T_{-\Psi(a_n)}$ is a M\"obius map of
the open ball in any $C^*$-algebra containing $A$, by the Schwarz-Pick equality of Harris
discussed above the theorem,  we have
$$\rho(\Phi(b_n), \Psi(b_n)) = \rho(\Phi(a_n), \Psi(a_n)) \to \infty.$$  
However $\Psi(b_n) = 0$, so that $\tanh^{-1} \| \Phi(b_n) \|  \to \infty$.
Thus $\| \Phi(b_n) \| \to 1$, contradicting (2).

(3) $\Rightarrow$ (4) \ If (4) fails then there exists $c \in (0,1)$ and a sequence $(a_n)$ in $A$ with $\| a_n \| < 1$ for 
all $n$ and $\| \Phi(a_n) \| \to 1$ but $\| \Psi(a_n) \| \leq c$.   By Lemma \ref{Harrt}, 
$\rho(\Phi(a_n), \Psi(a_n)) \to \infty,$   contradicting (3).

(5) $\Rightarrow$ (1) \ Suppose that (1) was false, but that $\tilde{\Phi} \leq \kappa \tilde{\Psi}$ for some $\kappa > 0$.
  Thus for any $\epsilon > 0$
there exists $f \in {\rm Ball}(A)$ with $\| \Phi(f) - \Psi(f) \| > 2-\epsilon$.   With $x = \Phi(f) - \Psi(f)$, and $t \in 
(0,1)$, choose similarly to an argument in the proof that (1) $\Rightarrow$ (2) above, a
contraction $u \in D$  such that
$u^* x = |x|^{2-t} \geq 0$.   So 
$\| u^* \Phi(f) - u^* \Psi(f) \| = \| x \|^{2-t} > (2-\epsilon)^{2-t}$.  
Then choose a 
state $\varphi$ of $D$  with $\varphi( \Phi(g) - \Psi(g)) > (2-\epsilon)^{2-t}$, where $g = u^* f + 1$.
Thus $${\rm Re} \, \varphi( \Phi(g) - \Psi(g)) 
= \varphi(   \Phi({\rm Re} \, g) - \Psi({\rm Re} \, g)) > (2-\epsilon)^{2-t}.$$
Note that $g$ is real positive in $A$, so that if $\psi$ is the state $\varphi \circ \Phi$ then Re $\psi(g) = \psi({\rm Re} \, g) \geq 0$.   Hence
$\alpha = \psi({\rm Re} \, g) 
  \geq 0$, and similarly
$\beta = \varphi(\Psi({\rm Re} \, g) ) 
\geq 0$.    Since
$\alpha, \beta \in [0,2]$ and $\alpha - \beta > (2-\epsilon)^{2-t}$, we must have $\alpha
> (2-\epsilon)^{2-t}$ and $\beta  < 2 - (2-\epsilon)^{2-t}$.
We also have
$${\rm Re} \, \Phi(g) = \tilde{\Phi}({\rm Re} \, g) \leq \kappa \, \tilde{\Psi}({\rm Re} \, g) = \kappa \,
 {\rm Re} \, \Psi(g) .$$
Hence $(2-\epsilon)^{2-t} < \alpha \leq \kappa \beta \leq \kappa (2 - (2-\epsilon)^{2-t}),$ whence $(2-\epsilon)^{2-t} < \frac{2 \kappa}{\kappa + 1}$.
Since $\epsilon > 0$ and $t \in (0,1)$ were arbitrary we have a contradiction.

(6) $\Rightarrow$ (5) \  By Lemma \ref{newl} and its proof, a real positive $D$-module map into $D$ is  real completely positive 
and  extends to a completely positive map $C \to B(H)$.  
Hence 
there exist completely positive maps
$\rho$ and $\nu$ on $C$ which extend
$\Phi - c \Psi$ and $\Psi - d \Phi$ respectively.
Thus on $A$ we have $\Phi = \rho + c \Psi
= \rho + c (\nu + d \Phi)$.  Evaluating at 1 we see that $1 \geq cd 1$.  
We may assume that $cd \neq 1$, or else $\rho + c \nu = 0$,
forcing $\rho = \nu = 0$ and $\Phi = c \Psi$ on $A$, which 
case is obvious.  Hence
$\tilde{\Phi} = (1 - cd)^{-1} (\rho + c \nu)$  extends $\Phi$ to $M$.
Similarly $\tilde{\Psi} = (1 - cd)^{-1} (\nu  + d \rho)$ extends $\Psi$ to 
$M$. Thus $d \tilde{\Phi} = (1 - cd)^{-1} (d \rho + cd \nu)
\leq \tilde{\Psi}$.   Similarly, $c \tilde{\Psi} \leq \tilde{\Phi}$.

For the remaining direction when $D = \Cdb$, assume that
(4) holds.    By the symmetry implied by the equivalence with (3), 
(4) also holds with $\Phi$ and $\Psi$ interchanged.  We first claim
that $\Phi - c \Psi$ is real positive on $A$
 for a positive constant $c$.
If $\Phi - c \Psi$ is not real positive for
any positive $c$, then there exist a sequence
$a_n \in {\mathfrak r}_A$ with Re $\Psi(a_n) = 1$
but Re $\Phi(a_n) \to 0$.  Replacing $a_n$ by
$h_n = \exp(-a_n)$ we have $\| h_n \| \leq 1$ since
$-a_n$ is dissipative.   Also, $$|\Psi(h_n)|
= | \exp(-\Psi(a_n)) | =
\exp(-{\rm Re} \, \Psi(a_n)) = e^{-1},$$
while  $|\Phi(h_n)| \to 1$
by a similar computation.  This contradicts
 (4).     Thus
$\Phi - c \Psi$ is real positive for a positive constant $c$.
Evaluating at $1$ shows that $c \in (0,1]$.    
Similarly, $\Psi - d \Phi$ is real positive on $A$
for a positive constant $d  \in (0,1]$.   Thus (6) holds.  \end{proof}

{\bf Remarks.}   (a) \ In (1) and (2) in the last theorem one may use the completely bounded norm.   That is,
the items are equivalent to $\| \Phi - \Psi \|_{\rm cb} < 2$ and to  $\| \Phi_{|{\rm Ker} \, \Psi} \|_{\rm cb} < 1$.
This may be proved by the `standard trick' described when we introduced $D$-characters.
Namely, recall that for $[x_{ij}] \in M_n(D)$ we have $\| [x_{ij} 
] \|_n = \sup \{ \| \sum_{i,j} s_i^* x_{ij} r_j \| \}$,
the supremum over $s_i, r_i \in D$ with $\sum_i \, r_i^* r_i$ and $\sum_i \, s_i^* s_i$ both contractive.
   Hence for $[x_{ij}] \in {\rm Ball}(M_n(A))$ then 
$\| [\Phi(x_{ij}) - \Psi(x_{ij}) ] \|$ is dominated by the supremum over such $r_i, s_i$ of 
$$\| \sum_{i,j} s_i^* (\Phi(x_{ij}) - \Psi(x_{ij})) r_j \| = \| \Phi(x)- \Psi(x) \| \leq \| \Phi - \Psi \|,$$
where $x = \sum_{i,j} s_i^* x_{ij} r_j \in {\rm Ball}(A)$.  Hence 
$\| \Phi - \Psi \|_{\rm cb} = \| \Phi - \Psi \|$, and by a similar argument
$\| \Phi_{|{\rm Ker} \, \Psi} \|_{\rm cb} = \| \Phi_{|{\rm Ker} \, \Psi} \|$.  \\

(b) \ Only  the proofs involving (1), and the 
proof that (2) implies (3),   
seem to require that $\Psi$ and $\Phi$ are $D$-module maps, etc.  The other proofs do require however that $\Phi, \Psi$ are contractive unital maps.   It is easy to find examples 
of contractive unital maps 
 where (2) holds but not (4), hence not (3). Also the $2 \times 2$ matrix example in the introduction 
to the present section, shows that then (1) cannot be equivalent to 
(3).

\medskip

Open question: Is (5)  equivalent to the other conditions  for $D$-characters?    
Note that only one of the two inequalities
in (5) was used to prove (1)--(4), so if (5) is equivalent to (1)--(4) then one  of the two inequalities
in (5) implies the other.

\begin{corollary} \label{iser}    Let $A$ be a unital operator algebra containing 
a  $C^*$-subalgebra $D$ unitally (with common identity).
 The equivalent  items  {\rm (1)}--{\rm (4)} in the last theorem define an equivalence relation on the
$D$-characters of $A$.  
\end{corollary} 

\begin{proof}   This follows since {\rm (4)} in the last theorem  is evidently an equivalence relation on the
$D$-characters.  
\end{proof}

We call the equivalence relation in the  corollary {\em Gleason equivalence}, and the 
  equivalence  classes will be called  {\em  Gleason parts}.   There has been much interest 
in the literature (\cite{SV,Bear}, etc) 
in the (sometimes coarser) equivalence relation 
often called {\em Harnack equivalence}, with the associated equivalence classes called  {\em  Harnack parts}. 
This is essentially the relation defined by (5) in Theorem \ref{Gleq}, but for more general $B(H)$-valued maps rather than our $D$-characters.  
Indeed the relation in (6) in the case that $A$ is the disk algebra and $D$ is one dimensional is the usual Harnack 
inequality for  harmonic functions on 
the disk.  We have not seen Gleason equivalence in our sense in the literature for operator algebras.

\medskip

Every Choquet boundary point $\omega$  in the maximal ideal space of a uniform algebra $A$
is a one point  Gleason part.  Indeed if $\chi$ is another character different from $\omega$ then 
there is a neighborhood $U$ of $\omega$ excluding $\chi$.   Also, by a basic characterization of Choquet points \cite{Gam} 
there is an $f \in {\rm Ball}(A)$  such that $f(\omega) = 1$ but $|f|< 1$ outside of $U$, and
in particular $|\chi(f)| < 1$.   
If $\chi$ were in the same part as $\omega$ this would violate {\rm (4)} in the last theorem.

For a noncommutative version of the fact in the last paragraph, suppose that either $\Phi : A \to D$ is a Choquet representation, or alternatively that $\varphi$ is a character
of $A$ which is an ${\mathfrak S}$-peaking state 
on $A$ in the sense of \cite[Section 4]{Clouatre}.    Clearly $\psi$ constitutes a one point  Gleason part by
 Theorem  \ref{Gleq} (4). 
For example on the upper triangular $2 \times 2$ matrices, 
evaluation at the $1$-$1$ corner  is a character
of $A$ which  is a peaking state, and is also a  one point  Gleason part in the maximal ideal space.
The support projection of this state in $A^{**}$ is a peak projection, and is a  minimal projection.

If $\varphi$ is a character of $A$ that admits a characteristic sequence in the sense of \cite[Section 4]{Clouatre}
then by Theorem 4.3 of \cite[Section 4]{Clouatre}, $\varphi$ is pure.  Combining 
\cite[Theorem 4.3 (1)]{Clouatre} and our Theorem  \ref{Gleq} (4) shows that  $\varphi$ 
constitutes a one point  Gleason part. 

\section{Application of Gleason parts to subdiagonal algebras} \label{hankel}

As an application of the theory of Gleason parts, we provide existence criteria for a so-called Wermer embedding function in the noncommutative context. As we shall subsequently illustrate, the importance of such a function lies in the fact that it ensures the existence of compact Hankel maps.
In this section  $A$ is a maximal 
subdiagonal subalgebra (in the sense of Arveson \cite{Arv}) of a von Neumann algebra $M$ with a faithful normal tracial state $\tau$.  We also assume that it is antisymmetric, 
that is $A \cap A^* = \Cdb \I$.

\begin{theorem}[Wermer embedding function] Let $A$ be antisymmetric,
 with $\omega$ a normal state in the Gleason part of $\tau$, distinct from $\tau$. Then there exists an element $z_r\in A_0$ which is invertible in $M$ such that $H^2(A)z_r=H^2_0(A)$. This element is of the form $h^{1/2}v_rh^{-1/2}$ for some unitary $v_r\in M$ where $h\in M_+^{-1}$ is the density for which $\omega=\tau(h\cdot)$. If in fact $\omega$ is also tracial, we have that $v_r$ 
commutes with $h$ and hence that $v_r=z_r$. 
Similarly, there exists an element $z_l\in A_0$ which is invertible in $M$ such that $z_lH^2(A)=H^2_0(A)$. This element is of the form $h^{-1/2}v_lh^{1/2}$ for some unitary $v_l\in M$. As before if $\omega$ is tracial, we have that $v_l=z_l$.
\end{theorem}

\begin{proof} We prove the existence of the element $z_r$. This proof will clearly also suffice to establish the existence of an element $w_r\in A_0^*$ of the form $h^{1/2}u_rh^{-1/2}$ for some unitary $u_r$, for which $H^2(A^*)w_r=H^2_0(A^*)$. 
The 
claim about $z_l$ and $v_l$ then follows by simply setting $z_l=w_r^*$
and $v_l = u_r^*$. For ease of notation we will drop the subscripts in the proof, and simply write $z$ and $v$ for $z_r$ and $v_r$.

Suppose that $\omega$ is in the Gleason part of $\tau$
 with $\omega\neq\tau$. 
The completion of $A$, $A_0$ and $M$ under the $L^2$-norm generated by $\omega$, will respectively be denoted by $H^2(\omega)$, $H^2_0(\omega)$ and $L^2(\omega)$. From the analysis of Gleason parts, we know that there exist $\alpha, \beta >0$ such that for all $g\in M^+$, $\alpha\tau(g)\leq \omega(g) \leq \beta\tau(g)$. This ensures that the spaces $H^2(\omega)$, $H^2_0(\omega)$ and $L^2(\omega)$ are effectively just equivalent renormings of $H^2(A)$, $H^2_0(A)$ and $L^2(M)$. 
The space $L^1(\omega)$ is similarly an equivalent renorming of $L^1(M)$. 
The action of the state $\omega$ admits a natural extension to the space $L^1(\omega)$, which we will still denote by $\omega$. It is an exercise to see that for this extension we have that $\omega(b^*a)=\langle a,b\rangle_\omega$ for all $a,b\in L^2(M)=L^2(\omega)$.  
Below we work in $L^2(M)$, regarding $L^2(\omega), H^2(\omega)$
and $H^2_0(\omega)$ as subspaces of $L^2(M)$ carrying a second norm and inner 
product. 

Now let $e \in H^2$ be the projection of $\I\in A$ onto $H^2_0(\omega)$ with respect to the inner product 
$\langle\cdot,\cdot\rangle_\omega$ coming from $\omega$. So $e\in H^2_0(\omega)$, with $\I-e$ orthogonal to $H^2_0$ in $L^2(\omega)$.
Observe that $c^2=\langle e,e\rangle_\omega=\|e\|^2_\omega\neq 0$, for otherwise $\I$ will be orthogonal to $A_0$ with respect to 
$\langle\cdot,\cdot\rangle_\omega$, which would in turn ensure that $\omega$ annihilates $A_0$. But that would force $\omega= \tau$, 
which would contradict our assumption. Hence we may let $z=\frac{1}{c}e$. 

Let $f\in A$ be given. Since then $fe\in H^2_0$, we have that $fe\perp_\omega(\I-e)$, and hence that $\omega(fe)=\omega(e^*fe)$. In particular for $f=\I$, we get $\omega(e)=c^2$. It is an exercise to see that the multiplicativity of $\omega$ on $A$ ensures that $\omega(ab)=\omega(a)\omega(b)$ for all $a\in A$, $b\in H^2$. From this it now follows that $$c^2\omega(f)=\omega(fe)=\omega(e^*fe) \quad\mbox{for all }f\in A.$$

We proceed to show that $c^2\omega(a)=\omega(|e|^2a)$ for all $a\in M$. To see this, firstly note that by construction, the functional $\gamma: M\to \mathbb{C}: a\to \frac{1}{c^2}\omega(e^*ae)$ is well-defined and positive on $M$, and assumes the value 1 at $\I$. Hence it is a state. It is however a state which agrees with $\omega$ on $A$. Therefore the claim follows by the noncommutative 
Gleason-Whitney theorem \cite[Theorem 4.2]{Ueda}.

Let $h\in L^1(M)_+$ be the density for which $\omega=\tau(h\cdot)$. The fact that there exist $\alpha, \beta >0$ such that for all $g\in M^+$, $\alpha\tau(g)\leq \omega(g) \leq \beta\tau(g)$, may alternatively be formulated as the claim that $\alpha\I \leq h \leq \beta\I$, or equivalently that $h\in M_+^{-1}$ as claimed.

The fact that for every $f\in M$ we have that $$\tau(hf)= \omega(f) =\frac{1}{c^2} \omega(e^*fe)=\frac{1}{c^2}\tau(he^*fe)=\frac{1}{c^2}\tau(ehe^*f),$$ensures that as affiliated operators of $M$, $h=\frac{1}{c^2}ehe^*$. This may be reformulated as the claim that $\I=\frac{1}{c^2}|h^{1/2}e^*h^{-1/2}|^2$. Since $M$ is finite, this in turn ensures that 
$v = \frac{1}{c}h^{-1/2}eh^{1/2}$ is a unitary element of $M$. It follows that $z=\frac{1}{c}e$ is of the form $z=h^{1/2}vh^{-1/2}$. In view of the fact that $H^2_0(A)=H^2_0(\omega)$, this description of $z$ moreover proves that $z\in H^2_0(A)\cap M=A_0$.   

Now observe that if $\omega$ is actually tracial, that would ensure 
that for any $a,b\in M$ we will have 
that $$\tau((ha)b)=\omega(ab)=\omega(ba)=\tau(hba)=\tau((ah)b).$$It follows 
that then $ha=ah$ for any $a\in M$, in other words $h\eta\mathcal{Z}(M)$. 
In this case we will therefore have that $z = h^{-1/2} v h^{1/2}= v = \frac{1}{c}e$.

It remains to prove that $H^2(A)z=H^2_0(A)$, or equivalently that $H^2(A)e=H^2_0(A)$. Since $e\in A_0$, it is clear that $H^2e\subset H^2_0$. Given that $e$ is an invertible element of $M$, $H^2e$ must be a closed subspace of $L^2(M)$. Let $g\in H^2_0\ominus_\omega H^2e$ be given. If we are able to show that we then necessarily have that $g=0$, it will follow that $H^2_0=H^2e$ as required. 

Since for any $f\in A$ we have that $fg\in H^2_0$, we will then also have that $(\I-e)\perp fg$ with respect to the inner product $\langle\cdot,\cdot\rangle_\omega$. In other words for any $f\in A$ we have that
$$0=\langle fg, (\I-e)\rangle_\omega=\omega(fg)-\omega(e^*fg).$$Next observe that $\omega(fg) =\omega(f)\omega(g)$ for any $f\in A$. To see this select any sequence $\{a_n\}\subset A$ converging to $g$ in the $L^2$-norm, and notice that we then have that $\omega(fg) =\lim_{n\to\infty}\omega(fa_n)=\lim_{n\to\infty}\omega(f).\omega(a_n)=\omega(f)\omega(g)$. Therefore $\omega(fg)=0$ for all $f\in A_\omega=\{a\in A: \omega(a)=0\}$. When combined with the previously centered equation, this ensures that$$0=\omega(e^*fg)=\omega((f^*e)^*g)= \langle g,f^*e\rangle_\omega\quad\mbox{for all }f\in A_\omega.$$

We have therefore shown that $g\perp (A+A_\omega^*)e$ with respect to the inner product $\langle\cdot,\cdot\rangle_\omega$. But $A+A_\omega^*=(A_0+\mathbb{C}\I)+A_\omega^*=A_0+(\mathbb{C}\I+A_\omega)^*=A_0+A^*$, and $A_0+A^*$ is known to be norm dense in 
$L^2(M)$. The equivalence of the norms generated by $\tau$ and $\omega$
 therefore ensures that $A_0+A^*$ is $\|\cdot\|_\omega$-dense in $L^2$. Since $e\in M^{-1}$, $(A_0+A^*)e$ is similarly $\|\cdot\|_\omega$-dense in $L^2$. Hence $g$ is orthogonal to $L^2$ with respect to $\langle\cdot,\cdot\rangle_\omega$, ensuring that $\|g\|_\omega=0$ as required.
\end{proof}

We proceed with illustrating the link of the above theorem to the existence of compact Hankel maps. 

\begin{definition}
Given $\varphi\in M$ we define the Hankel map $\mathcal{H}_f$ with symbol $f$ to be the map $\mathcal{H}_f:H^2(A)\to(H^2_0(A))^*: a\mapsto (\I-P_+)(f a)$, where $P_+$ is the orthogonal projection of $L^2(M)$ onto $H^2(A)$.
\end{definition}

The following results from \cite{LX-Toeplitz} are crucial  in establishing necessary conditions for the existence of compact Hankel maps.  The second 
result is  a faithful non-commutative version of \cite[Theorem 2]{CMNX},
 and its proof closely follows the proof in \cite{CMNX}.

\begin{lemma}[\cite{LX-Toeplitz}] Suppose there exists an element $z\in A_0$ which is unitary in $M$, such that $zH^2(A)=H^2_0(A)$. Then the left multiplication operators $\{M_{z^n}\}\subset B(L^2(M))$ (respectively $\{M_{(z^*)^n}\}$) converge to 0 in the weak operator topology.
\end{lemma}

\begin{theorem}[Existence  of compact Hankel maps \cite{LX-Toeplitz}]  \label{LX}
Let $A$ be antisymmetric and suppose that there exists an element $z\in A_0$ (invertible in $M$) such that $zH^2(A)=H^2_0(A)$. Given $f\in M$, the Hankel map $\mathcal{H}_f$ will be compact if $f$ belongs to the norm closed subalgebra generated by $z^{-1}$ and $A$. If indeed $z$ is unitary in $M$, then whenever $\mathcal{H}_f$ is compact, $f$ will conversely necessarily belong to the norm closed subalgebra generated by $z^{-1}=z^*$ and $A$ .
\end{theorem}

As mentioned in the introduction, the significance of this result, and hence by extension the importance of the Wermer embedding function, can only be appreciated if one notes that there are some group algebraic contexts which admit no non-trivial compact Hankel matrices whatsoever! (See the discussion preceding Definition 10 in \cite{Exel}.) We provide a sketch of the proofs of these 
two results - full details may be found in \cite{LX-Toeplitz}.  
For the lemma, the first step  is to notice that the unitarity of $z$ combined with the fact that $z\in A_0$, ensures that $\{z^n\}$ is an  orthonormal system in $L^2(M)$. Hence $z^n\to 0$ weakly in $L^2(M)$. That in turn ensures that for any $f\in L^2(M)$, $z^nf\to 0$ weakly in $L^1(M)$. The crucial step in the proof is to show that for any $f\in L^2(M)$, we in fact have that $z^nf\to 0$ weakly in $L^2(M)$. The sequence $\{z^nf\}$ is easily seen to to be norm bounded in $L^2(M)$, and so the proof of this fact consists of showing that 0 is the only weak limit point of this sequence. The final step is to check that weak-$L^2$ convergence of $z^nf$ to 0, ensures convergence to 0 of $\{M_{z^n}\}$ in the weak operator topology.

The first claim in Theorem \ref{LX} essentially consists of showing that for any polynomial $p$ in $z^{-1}$ and finitely many elements of $A$, $\mathcal{H}_p$ is finite rank and hence compact. The compactness of $\mathcal{H}_f$ will then follow upon checking that $\mathcal{H}_f$ is the norm limit of such finite rank operators $\mathcal{H}_p$. (This is essentially the same argument as the one followed in \cite[Theorem 2]{CMNX}.) 

To prove the second statement some preparation is necessary. Let $f\in M$ be given.
We first note that a modification of \cite[Lemma 4.5]{LX-Helson} shows that 
$$\|\mathcal{H}_f\| = \sup\{|\tau(fF)|: F\in H^1_0(M), \tau(|F|)\leq 1\}.$$(Since the definition of Hankel maps in \cite{LX-Helson} differs slightly from the one presented here, some checking is necessary.) The next step is to notice that the duality argument in the last part of the proof of \cite[Theorem 3.9]{LX-Helson} suffices to show that $$\sup\{|\tau(fF)|: F\in H^1_0(M), \tau(|F|)\leq 1\} = \inf\{\|f+a\|_\infty : a \in A a\}.$$Combining these observations now yields the fact that $\|\mathcal{H}_f\|=\inf\{\|f+a\|_\infty : a \in A a\}$. This fact, together with the Lemma above, now provides us with all the technology required for the proof of \cite[Lemma 2.3]{CMNX} to go through almost verbatim in the present setting. 

\bigskip

{\em Acknowledgment.}   We thank Sanne Ter Horst for conversations and references on the hyperbolic metric on Harnack parts
and Brian Simanek  for much information on the balayage of probability measures.  We are also indebted to the referees for their 
comments, which much improved the presentation of the paper.

\end{document}